\documentclass[11pt]{article}

\usepackage{amsmath,amsfonts,amssymb,amsthm}
\usepackage{enumerate}
\usepackage{xcolor}
\usepackage{url}
\usepackage{tcolorbox}
\usepackage{hyperref}

\usepackage{multicol, latexsym}
\usepackage{latexsym}
\usepackage{psfrag,import}
\usepackage{verbatim}
\usepackage{color}
\usepackage{epsfig}
\usepackage[outdir=./]{epstopdf}
\usepackage{hyperref}
\hypersetup{
    colorlinks=true,
    linkcolor=blue,
    filecolor=magenta,      
    urlcolor=cyan
}
\usepackage[title]{appendix}
\usepackage{geometry}
\usepackage{mathtools}
\usepackage{enumerate}
\usepackage{enumitem}   
\usepackage{multicol}
\usepackage{booktabs}
\usepackage{enumitem}
\usepackage{lineno}
\usepackage{parcolumns}
\usepackage{thmtools}
\usepackage{xr}
\usepackage{epstopdf}
\usepackage{mathrsfs}
 \usepackage{subcaption}
 \usepackage{soul}
\usepackage{float}

\usepackage{comment}
\usepackage{authblk}
\usepackage{setspace}

\usepackage{cleveref}

\theoremstyle{plain}
\newtheorem{theorem}{Theorem}[section]
\newtheorem{corollary}[theorem]{Corollary}
\newtheorem{proposition}[theorem]{Proposition}

\newtheorem{lemma}[theorem]{Lemma}

\theoremstyle{definition}
\newtheorem{definition}[theorem]{Definition}

\newtheorem{remark}[theorem]{Remark}
\newtheorem{example}[theorem]{Example}

\theoremstyle{remark}

\newcommand{\mK}{\mathcal{K}}
\newcommand{\dK}{\mathcal{K}_{\text{disg}}}
\newcommand{\gK}{\mathcal{K}_{\text{global}}}

\newcommand\R{\mathbb{R}}

\newcommand\bW{\boldsymbol{W}}

\newcommand\bof{\boldsymbol{f}}
\newcommand\bx{\boldsymbol{x}}
\newcommand\bw{\boldsymbol{w}}
\newcommand\bk{\boldsymbol{k}}
\newcommand\bJ{\boldsymbol{J}}

\newcommand\by{\boldsymbol{y}}

\newcommand\bv{\boldsymbol{v}}

\newcommand\bzeta{\boldsymbol{\zeta}}

\usepackage{xkeyval}
\makeatletter
\newlength{\LP@lh@boxwidth}      
\newlength{\LP@lh@boxsep}        
\newcommand{\LP@lh@content}{}
\newcommand{\LP@lh@tagname}{}    
\define@key[LP]{lhbox}{boxsep}[2em]{\setlength{\LP@lh@boxsep}{#1}}
\define@key[LP]{lhbox}{content}[]{\renewcommand{\LP@lh@content}{#1}}
\define@key[LP]{lhbox}{tagname}[]{\renewcommand{\LP@lh@tagname}{#1}}
                         {\setkeys[LP]{lhbox}{tagname,boxsep,#1}\relax%
                          \settowidth{\LP@lh@boxwidth}{\tagform@\LP@lh@tagname}%
                          \noindent%
                          \parbox{\LP@lh@boxwidth}{\begin{align}\tag{\LP@lh@tagname}\LP@lh@content\end{align}}%
                          \hspace*{\LP@lh@boxsep}%
                          \begin{minipage}{\linewidth-\LP@lh@boxwidth-\LP@lh@boxsep}}%
                         {~\end{minipage}}
\makeatother

\newcommand{\New}{\textbf{New}}

\newcommand{\defi}[1]{\textbf{#1}}
\newcommand{\spn}{\text{span} \ }

\graphicspath{images/ReactionNetworkExamples//}
        {\begin{list}
                {\noindent\makebox[0mm][r]{$\bullet$}}
                {\leftmargin=5.5ex \usecounter{enumi}
      \topsep=1.5mm \itemsep=-.75ex}
        }
        {\end{list}}
        
\title{Realizations through Weakly Reversible Networks and the Globally Attracting Locus}

\author[1]{Samay Kothari}
\author[2]{Jiaxin Jin}
\author[3]{Abhishek Deshpande}
\affil[1,3]{\small Center for Computational Natural Sciences and Bioinformatics, International Institute of Information Technology Hyderabad} 
\affil[2]{\small Department of Mathematics, University of Louisiana at Lafayette}

\begin{document}

\maketitle

\begin{abstract}

We investigate the possibility that for any given reaction rate vector $\bk$ associated with a network $G$, there exists another network $G'$ with a corresponding reaction rate vector that reproduces the mass-action dynamics generated by $(G, \bk)$.
Our focus is on a particular class of networks for $G$, where the corresponding network $G'$ is weakly reversible.  
In particular, we show that strongly endotactic two-dimensional networks with a two-dimensional stoichiometric subspace, as well as certain endotactic networks under additional conditions, exhibit this property.
Additionally, we establish a strong connection between this family of networks and the locus in the space of rate constants of which the corresponding dynamics admits globally stable steady states.
\end{abstract}


\section{Introduction}
\label{sec:intro}

Dynamical systems with polynomial right-hand sides can be represented in the form:
\begin{equation} \notag
\begin{split}
    \frac{dx_1}{dt} &= f_1(x_1, \ldots, x_n), \\ 
    \frac{dx_2}{dt} &= f_2(x_1, \ldots, x_n), \\ 
                    &\qquad \quad \vdots  \\
    \frac{dx_n}{dt} &= f_n(x_1, \ldots, x_n), \\ 
\end{split}
\end{equation}
where each $f_i(x)$  is a polynomial function of the variables $x_1, \ldots, x_n$. 
Under assumptions of well-mixedness and a large number of reactants, systems arising from real-life applications can be modeled using the polynomial system described above.
These include, in particular, chemical and biological interaction networks that exhibit diverse behaviors such as multistability \cite{craciun2005multiple,craciun2006multiple}, bifurcations \cite{wilhelm1995smallest,wilhelm1996mathematical}, oscillations~\cite{boros2024oscillations}, homeostasis~\cite{craciun2022homeostasis}, and chaotic dynamics \cite{erdi1989mathematical}. 

The classical theory of reaction networks, developed by Horn, Jackson, and Feinberg~\cite{feinberg1979lectures,feinberg1987chemical,feinberg2019foundations,horn1972general}, offers a rigorous framework for understanding the relationship between the dynamics of these networks and their underlying structure. For instance, properties such as \emph{persistence} \cite{craciun2013persistence}, \emph{permanence}, and the existence of a \emph{global attractor} \cite{craciun2009toric,ding2021minimal,ding2022minimal} are often intrinsically linked to the structural features of the networks. 
It is important to note that reaction networks can be analyzed from both deterministic~\cite{anderson2011boundedness,anderson2011proof,anderson2010dynamics,conradi2007subnetwork,craciun2020efficient,hell2015dynamical,yu2018mathematical} and stochastic perspectives~\cite{anderson2016product, anderson2015lyapunov}. However, this paper will focus on the deterministic approach.

A notable phenomenon in this context is that different reaction networks can produce identical dynamics due to a concept known as \emph{dynamical equivalence} \cite{deshpande2023source,craciun2020efficient,craciun2023weakly,craciun2023weaklysingle}. It is also referred to as \emph{macro-equivalence} \cite{horn1972general} or \emph{confoundability} \cite{craciun2008identifiability}. 
Consider the following two reaction networks:
\begin{equation} \notag
\begin{split}
& \text{Network 1: \qquad }
\emptyset \xrightarrow[]{k_1} X + Y, \qquad
X + Y \xrightarrow[]{k_2} 2X + Y, \qquad
2X + Y \xrightarrow[]{k_3} \emptyset
\\[5pt]
& \text{Network 2: \qquad }
\emptyset \xrightarrow[]{k_1} X, \qquad
\emptyset \xrightarrow[]{k_1} Y, \qquad
X + Y \xrightarrow[]{k_2} 2X + Y, \qquad
2X + Y \xrightarrow[]{k_3} \emptyset
\end{split}
\end{equation}
Both networks generate the same dynamical system as follows:
\begin{equation}
\begin{split} \notag 
\frac{dx}{dt} & = k_1 + k_2 xy - 2k_3x^2y, \\
\frac{dy}{dt} & = k_1 -  k_3x^2y.  
\end{split}
\end{equation}
However, the networks are structurally different. 
\emph{Network 1 consists of three complexes and a single linkage class, while Network 2 consists of five complexes and three linkage classes}. The structural properties of Network 2 do not align with the theorems from classical reaction network theory, making it difficult to infer the long-term behavior of its dynamics. In contrast, Network 1 is weakly reversible and has a single linkage class, suggesting that its dynamics are permanent and possess globally attracting steady states. Given that both networks are dynamically equivalent, the dynamical system produced by Network 2 shares these properties.

Certain classes of networks, such as \emph{weakly reversible}, \emph{endotactic}, and \emph{strongly endotactic} networks, are known for their robust dynamical properties. Weak reversibility indicates that every connected component of the network is strongly connected. Informally, a network is termed endotactic if its reactions ``point inward'' \cite{craciun2013persistence}. If the network meets certain additional criteria, then it is classified as strongly endotactic \cite{gopalkrishnan2014geometric}. 
In two dimensions, endotactic networks are known to be permanent \cite{craciun2013persistence}, while strongly endotactic networks in any dimension are permanent \cite{gopalkrishnan2014geometric}. Furthermore, weakly reversible networks with a single connected component possess globally attracting steady states~\cite{anderson2011proof}.
Given these classes of networks and their properties, we address the following question:

\medskip

\textbf{Question:} Are there dynamical inclusions between weakly reversible, endotactic, and strongly endotactic networks?

\medskip

This work is part of a broader effort to comprehensively characterize the relationships between different families of reaction networks. We answer the above question affirmatively, showing that, under specific structural conditions, the dynamics of endotactic and strongly endotactic networks can be realized by weakly reversible networks. 
Furthermore, we investigate a class of dynamical systems generated by weakly reversible networks, referred to as \emph{complex-balanced systems}, for which the steady states are conjectured to be globally stable. We then establish their connections to \emph{disguised toric systems}, which are dynamically equivalent to complex-balanced systems, and provide sufficient conditions for a reaction network to have a non-empty locus in the space of rate constants that generate globally stable steady states.

\medskip

\textbf{Structure of the Paper:}
The paper is organized as follows: In Section~\ref{sec:reaction_networks}, we review some fundamental definitions related to reaction networks. 
Section~\ref{sec:dyn_equiv} introduces the concept of dynamical equivalence, a phenomenon where two distinct reaction networks can generate the same dynamical system. 
In Section~\ref{sec:necessary}, we demonstrate that an endotactic network is necessary for its dynamics to be included within that of a weakly reversible network. 
Section~\ref{sec:strongly_endo_two} shows that the dynamics of a strongly endotactic network with a two-dimensional stoichiometric subspace can be realized by the dynamics of a weakly reversible network under a different set of assumptions. 
In Section~\ref{sec:general_strong_endo}, we generalize this result to endotactic networks in any dimensions under certain assumptions.
Section~\ref{sec:toric_global_locus} defines the notion of the \emph{globally stable locus} and provides sufficient conditions for this locus to be non-empty.
In Section~\ref{sec:application_positive}, we provide examples from biological networks and apply the theorems from Section~\ref{sec:sufficient} to show that the globally stable locus of these networks is non-empty.
Finally, in Section~\ref{sec:disc}, we summarize our findings and suggest directions for future research.


\section{Basic Terminology in Reaction Networks}
\label{sec:reaction_networks}

In this section, we briefly introduce reaction networks and discuss some of their important properties. Our exposition follows~\cite{feinberg1979lectures}.

\begin{definition}

\begin{enumerate}[label=(\roman*)]
\item A \defi{reaction network} $G = (V, E)$ can be described as a finite directed graph embedded in Euclidean space $\mathbb{R}^n$, also called an \defi{E-graph}, where $V \subset \mathbb{R}^n$ represents a finite set of \defi{vertices} without isolated vertices, and $E \subseteq V \times V$ represents a finite set of \defi{edges} with no self-loops and at most one edge between a pair of ordered vertices.

\item A directed edge $(\by,\by')\in E$, denoted by $\by \to \by'\in E$, is called a \defi{reaction} in the network. 
For every reaction $\by \to \by'\in E$, $\by$ is called the \defi{source vertex} and $\by'$ is called the \defi{target vertex}, and $\by'-\by$ denotes the corresponding \defi{reaction vector}. Moreover, we let $V_s$ denote the set of source vertices of $G$.
\end{enumerate}
\end{definition}

\begin{definition}

Let $G=(V, E)$ be an E-graph.

\begin{enumerate}[label=(\roman*)]
\item Every connected component of $G$ is called a \defi{linkage class} of $G$. Further, we denote each linkage class by the set of vertices it contains.
    
\item A connected component $V' \subseteq V$ is said to be \defi{strongly connected} if every edge is part of a cycle.  
$G$ is said to be \defi{weakly reversible} if every connected component of $G$ is strongly connected.
    
\item A strongly connected component $V' \subseteq V$ is said to be \defi{terminal}, if for any reaction $\by \to \by' \in E$ with $\by \in V'$ then $\by' \in V'$.
\end{enumerate}
\end{definition}

\begin{definition}[\cite{sontag2001structure}]

Let $G=(V, E)$ be an E-graph.

\begin{enumerate}[label=(\roman*)]
\item The \defi{stoichiometric subspace} of $G$, denoted by $S_{G}$, is given by 
\begin{eqnarray} \notag
S_G = \spn \{ \by'-\by \ \mid \ \by \to \by' \in E \}.
\end{eqnarray}

\item Given $\bx_0 \in \mathbb{R}^n_{>0}$, the \textbf{stoichiometric compatibility class} of $\bx_0$ is given by $(\bx_0 + S_G)\cap\mathbb{R}^n_{>0}$.
Further, if $V \subset \mathbb{Z}_{\geq 0}^n$, then the stoichiometric compatibility class of $\bx_0$ is forward-invariant.
\end{enumerate}
\end{definition}

\begin{definition} \label{def:newton}

Let $G=(V, E)$ be an E-graph. The \textbf{Newton polytope} of $G$, denoted by $\New (G)$, is the convex hull formed by the source vertices $V_s$ of $G$.
\end{definition}

\begin{figure}[!ht]
    \centering
    \begin{subfigure}[t]{0.5\textwidth}
        \centering
        \includegraphics[width = 0.4\textwidth]{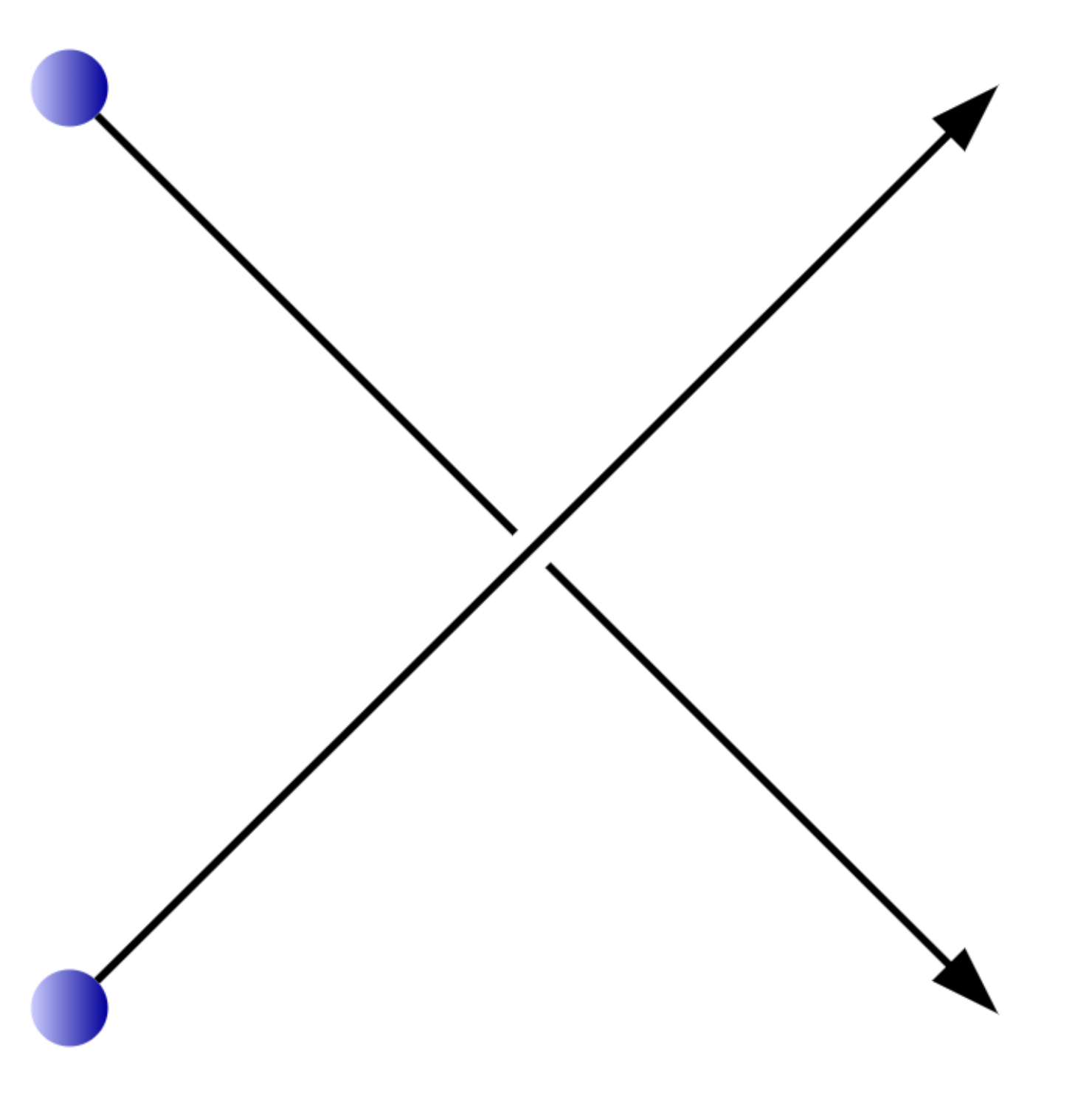}
        \caption{}
    \end{subfigure}%
    \begin{subfigure}[t]{0.5\textwidth}
        \centering
        \includegraphics[width = 0.4\textwidth]{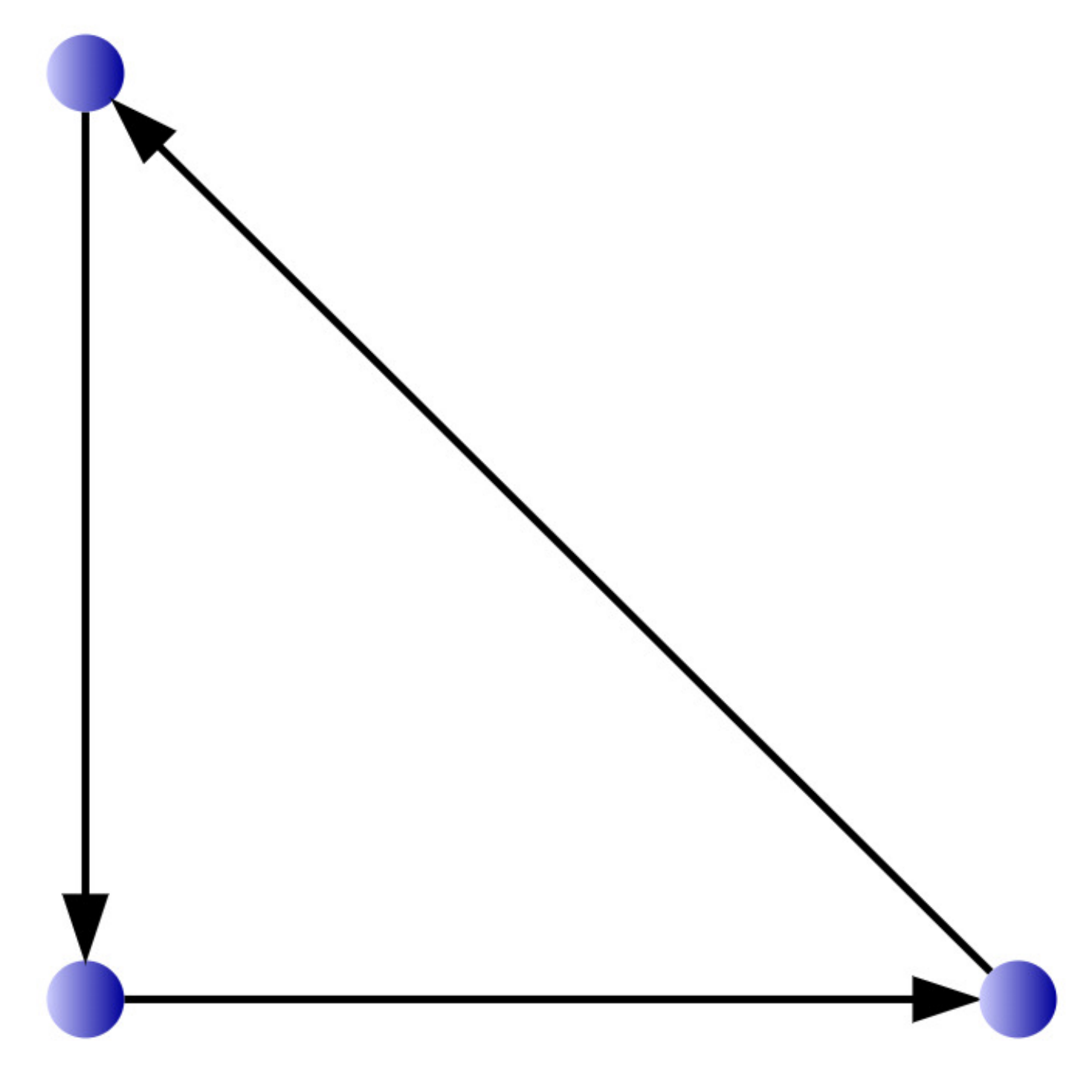}
       \caption{}
    \end{subfigure}%
    \\
    \begin{subfigure}[t]{0.5\textwidth}
        \centering
        \includegraphics[width = 0.4\textwidth]{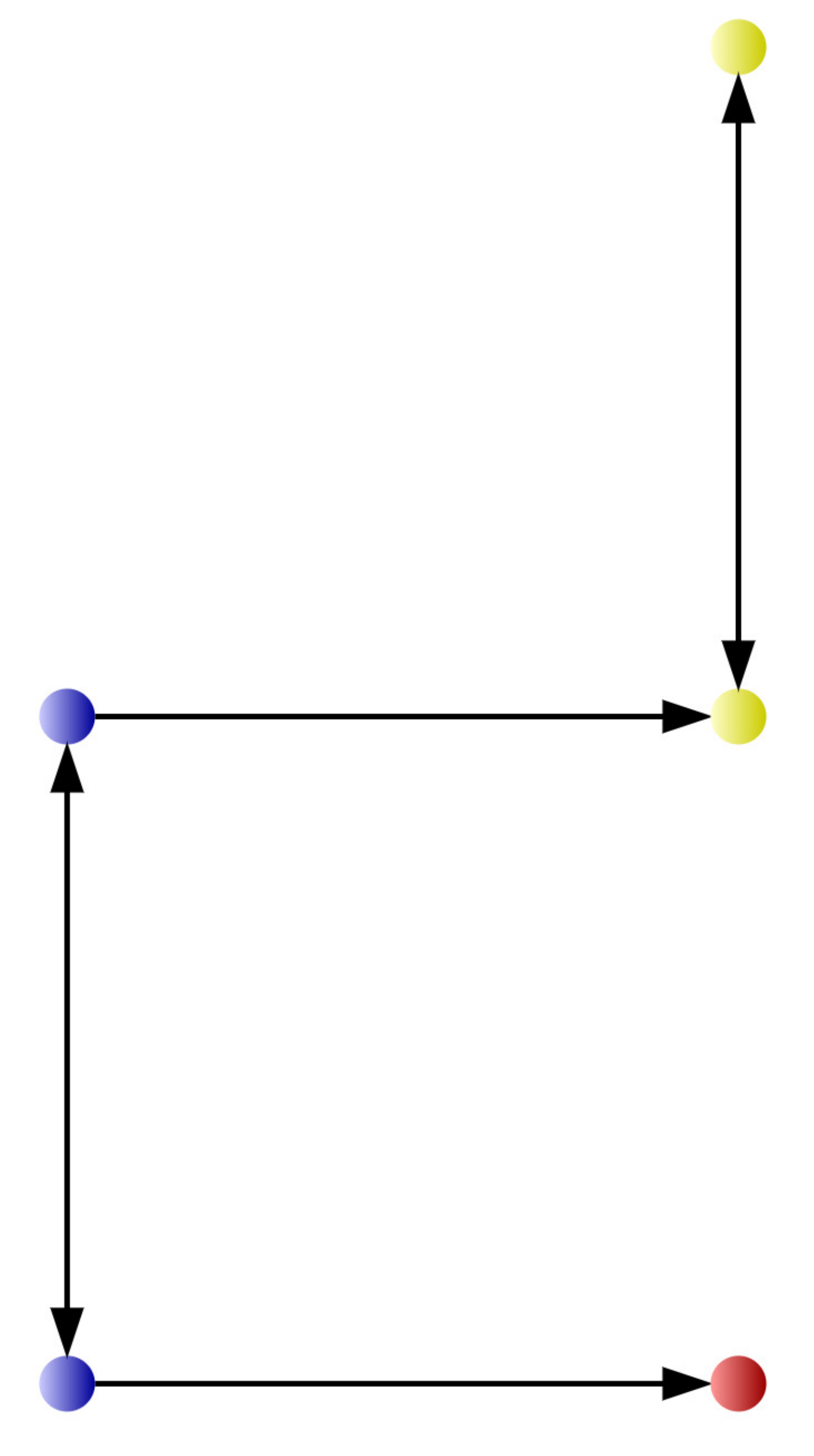}
       \caption{}
    \end{subfigure}%
    \begin{subfigure}[t]{0.5\textwidth}
        \centering
        \includegraphics[width = 0.4\textwidth]{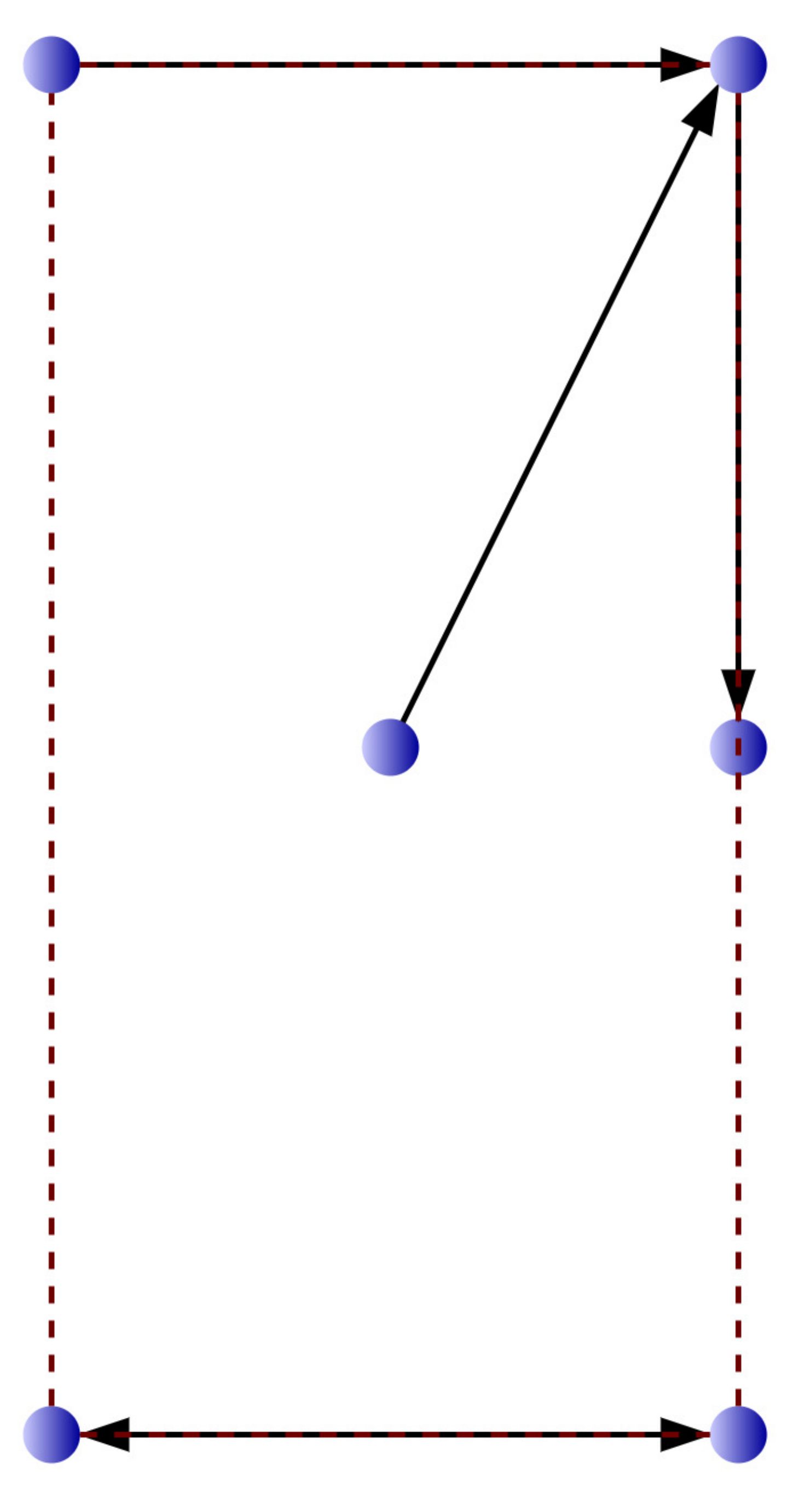}
       \caption{}
    \end{subfigure}%
    \caption{(a) An E-graph . (b) A weakly reversible E-graph. (c) Each color in the network encodes a strong linkage class. The component comprising yellow vertices is a terminal strong linkage class. In addition, the component comprising a single red vertex is also a terminal strong linkage class. The component comprising of blue vertices is a non-terminal strong linkage class (d) The rectangular region   represents the Newton Polytope of the given network (denoted by \New (G)).}
\end{figure}

\begin{definition}

Let $G=(V, E)$ and $G' = (V', E')$ be two E-graphs. 
\begin{enumerate}[label=(\roman*)]
\item $G$ is called a \defi{complete graph} if $\by \to \by'\in E$ for every pair of vertices $\by, \by' \in V$.
Further, a complete graph can be obtained by connecting every pair of vertices in $G = (V, E)$, denoted by $G_c = (V, E_c)$, which is called the \defi{complete graph on $G$}.

\item $G$ is called a \defi{subgraph of $G'$}, denoted by $G\subseteq G'$, if $V \subseteq  V'$ and $E \subseteq E'$. 
Further, we let $G \subseteq_{wr} G'$ denote that $G$ is a weakly reversible subgraph of $G'$.
\end{enumerate}
\end{definition}

\begin{definition}[\cite{craciun2013persistence, gopalkrishnan2014geometric}]
\label{def:endotactic}

Let $G=(V, E)$ be an E-graph. Then
\begin{enumerate}[label=(\roman*)]
\item $G$ is \textbf{endotactic}, if for every reaction $\by \rightarrow \by' \in E$ and $\bv \in \R^n$ satisfying $\bv \cdot (\by'-\by) < 0$, there exists $\tilde{\by} \rightarrow \tilde{\by}' \in E$ such that 
\begin{equation} \label{eq:endotactic}
\bv \cdot \tilde{\by} < \bv \cdot \by 
\ \text{ and } \ 
\bv \cdot (\tilde{\by}'- \tilde{\by}) > 0.
\end{equation}

\item An endotactic E-graph $G$ is said to be \textbf{strongly endotactic}, if for every reaction $\by\rightarrow\by'\in E$ and $\bv\in\R^n$ satisfying $\bv\cdot(\by'-\by)<0$, there exists $\tilde{\by}\rightarrow\tilde{\by}'\in E$ satisfying \eqref{eq:endotactic} and
\[
\bv \cdot \tilde{\by} \leq \bv\cdot\bar{\by}
\ \text{ for every } \
\bar{\by}\in V_{s}.
\]
\end{enumerate}
\end{definition}

\begin{definition}[\cite{adleman2014mathematics,feinberg1979lectures,guldberg1864studies,gunawardena2003chemical,voit2015150,yu2018mathematical}]
\label{def:massAction}

Let $G=(V, E)$ be an E-graph. 
The \defi{reaction rate vector} of $G$ is given by 
\begin{equation} \notag
\bk = (k_{\by \rightarrow \by'})_{\by \rightarrow \by' \in E} \in \mathbb{R}^{|E|}_{>0},
\end{equation}
where $k_{\by \rightarrow \by'} > 0$ is called the \defi{reaction rate constant} of the reaction $\by \rightarrow \by' \in E$. 
Further, the \defi{associated mass-action system generated by $(G, \bk)$} is given by
\begin{eqnarray} \label{eq:mass_action}
\frac{d\bx}{dt} = \displaystyle\sum_{\by\rightarrow \by'} k_{\by\rightarrow\by'}{\bx}^{\by}(\by'-\by).
\end{eqnarray} 
\end{definition}

\begin{definition}[\cite{horn1972general}]
\label{def:cb}

Let $(G,\bk)$ be a mass-action system. Then $(G, \bk)$ is said to be \textbf{complex-balanced}, if there exists $\bx_0 \in \mathbb{R}^n_{>0}$ such that for every vertex $\by_0 \in V$,
\begin{eqnarray} \notag
\displaystyle\sum_{\by_0\rightarrow \by'\in E}{\bk}_{\by_0\rightarrow \by'}{\bx_0}^{\by_0} = \displaystyle\sum_{\by'\rightarrow \by_0\in E}{\bk}_{\by'\rightarrow \by_0}{\bx_0}^{\by'}.
\end{eqnarray}
\end{definition}

\begin{remark}[\cite{craciun2013persistence,gopalkrishnan2014geometric,horn1972general}]
\label{rmk:cb_property}

If $(G, \bk)$ is a complex-balanced system, then the E-graph $G$ is weakly reversible and thus it is endotactic. Further, every weakly reversible network consisting of a single linkage class is strongly endotactic and every strongly endotactic network is endotactic.
\end{remark}

\begin{figure}[!ht]
    \centering
    \begin{subfigure}[t]{0.33\textwidth}
        \centering
        \includegraphics[width = 0.6\textwidth]{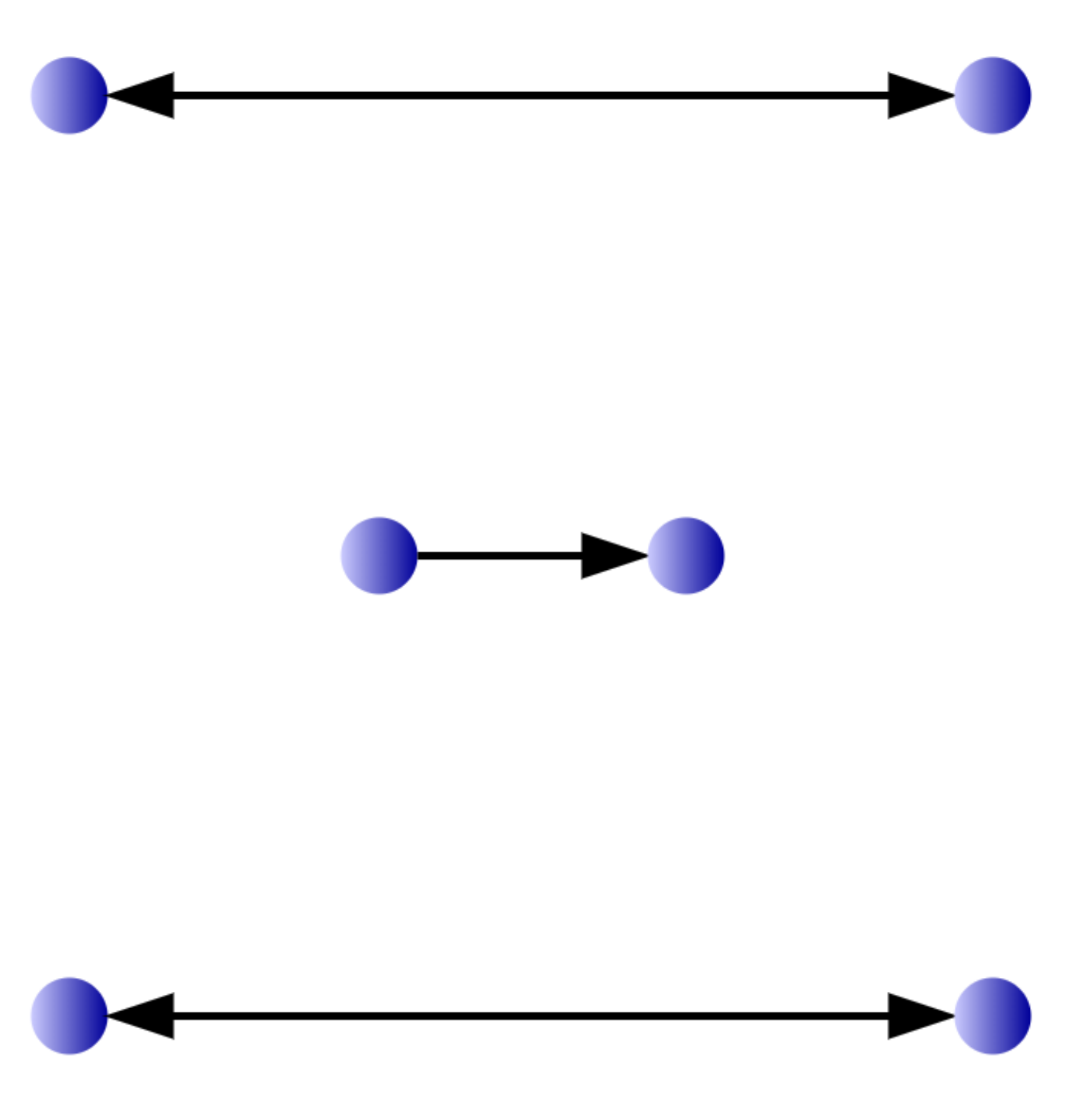}
       \caption{}
    \end{subfigure}%
    \qquad
    \begin{subfigure}[t]{0.33\textwidth}
        \centering
        \includegraphics[width = 0.6\textwidth]{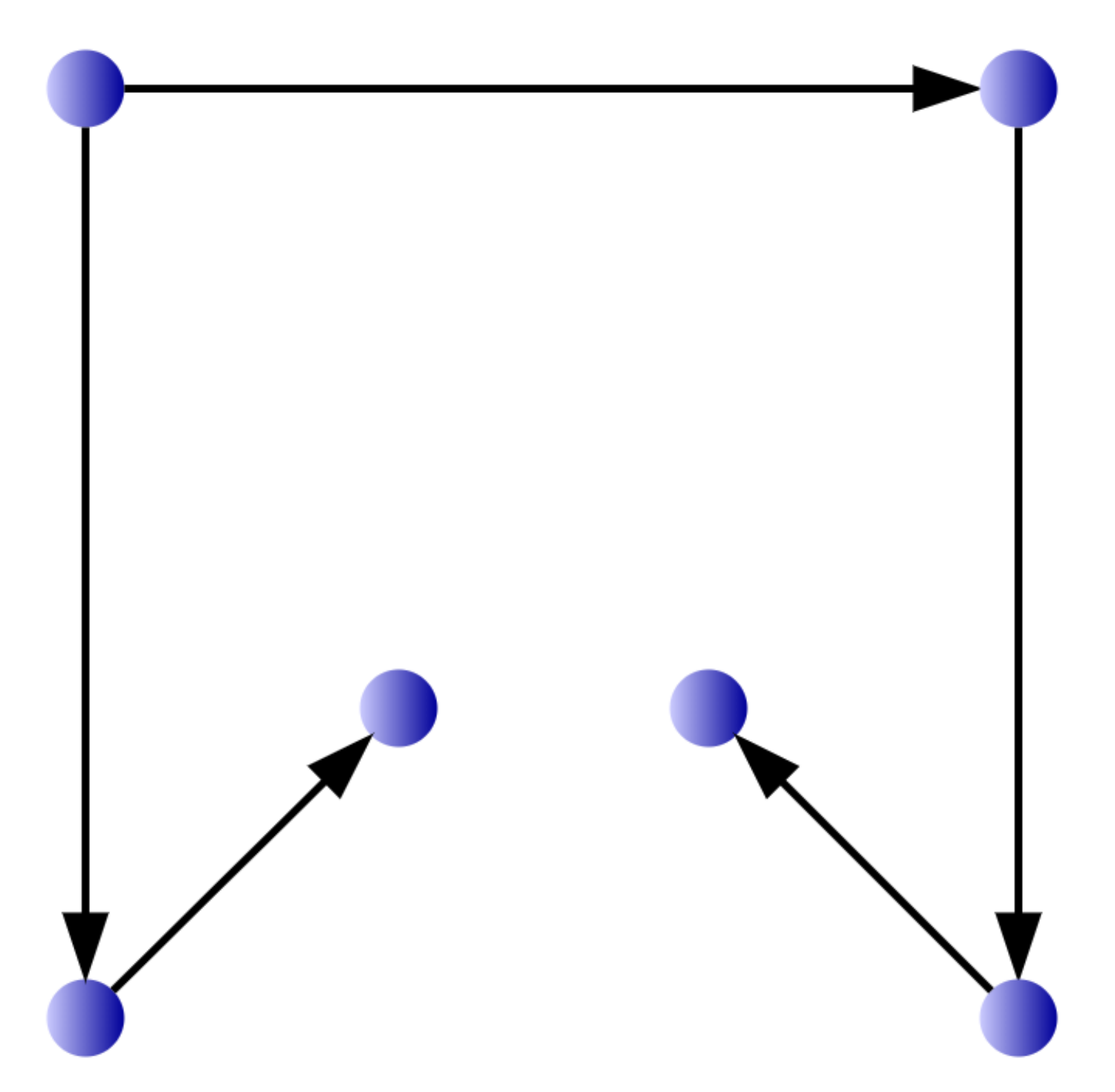}
       \caption{}
    \end{subfigure}%
    \caption{(a) An endotactic E-graph. (b) A strongly endotactic E-graph.}
\end{figure}

\begin{definition}

Let $(G,\bk)$ be a mass-action system. Then $(G,\bk)$ is called \textbf{persistent} if for any initial condition $\bx(0)\in\mathbb{R}^n_{>0}$, the solution $\bx(t)$ satisfies that
\begin{eqnarray} \notag
\lim_{t\to T}\inf \bx_i(t) > 0 
\ \text{ for every } \
i = 1, \ldots, n,
\end{eqnarray} 
where $T$ is the maximum time the solution $\bx(t)$ is well-defined.
\end{definition}

\begin{definition}

Let $(G,\bk)$ be a mass-action system. Then $(G,\bk)$ is called \textbf{permanent}
if for any initial condition $\bx(0)\in\mathbb{R}^n_{>0}$, there exists a compact set $\mathcal{P} \in (\bx_0 + S_G)\cap\mathbb{R}^n_{>0}$ such that $\bx(t)\in \mathcal{P}$ eventually.
\end{definition}

\begin{definition}

Let $(G,\bk)$ be a mass-action system.
A point $\bx^* \in\mathbb{R}^n_{>0}$ is said to be a \textbf{global attractor} 
if for every $\bx(0)\in (\bx^* + S_G)\cap\mathbb{R}^n_{>0}
$, then $\displaystyle\lim_{t\to\infty} \bx(t) = \bx^*$. 
\end{definition}

These properties are related to the following open problems:

\smallskip

\textbf{Persistence Conjecture:} Weakly reversible dynamical systems are persistent. 

\smallskip

\textbf{Global Attractor Conjecture:} Complex-balanced dynamical systems have a globally attracting steady state within each stoichiometric compatibility class.

\smallskip

Some partial cases of these conjectures have been established. Angeli, Sontag, and Leenheer have shown that weakly reversible dynamical systems without critical siphons are persistent~\cite{angeli2007petri}. This connection has been further strengthened by Gopalkrishnan and Deshpande~\cite{deshpande2014autocatalysis}, by giving a combinatorial characterization of critical siphons in terms of autocatalytic sets~\cite{craciun2022autocatalytic}. 

Anderson~\cite{anderson2011proof} has proved the Global Attractor Conjecture for weakly reversible networks consisting of a single linkage class. Craciun, Nazarov, and Pantea~\cite{craciun2013persistence} have proved the persistence conjecture for endotactic networks in two dimensions. This has been extended by Pantea~\cite{pantea2012global} to endotactic networks with two-dimensional stoichiometric subspace. Gopalkrishnan, Miller, and Shiu~\cite{gopalkrishnan2014geometric} have proved the permanence of strongly endotactic networks. Craciun~\cite{craciun2015toric} has proposed a proof of the Global Attractor Conjecture using the idea of toric differential inclusions\cite{craciun2019polynomial,craciun2019quasi,craciun2020endotactic}. More recently, Craciun and Deshpande have shown that endotactic networks can be embedded into toric differential inclusions~\cite{craciun2020endotactic}. Further, an alternate characterization of toric differential inclusions has been obtained in~\cite{craciun2019quasi}.

\section{Dynamical Equivalence}
\label{sec:dyn_equiv}

To account for the possibility that different reaction networks can generate the same dynamics, the notion of \emph{dynamical equivalence} was introduced. 
This concept is also referred to as \emph{macro-equivalence} \cite{horn1972general} and \emph{confoundability}~\cite{craciun2008identifiability}.

\begin{definition} 

Consider a dynamical system of the form 
\begin{equation} \label{eq:xxx}
\frac{d\bx}{dt} = \bof (\bx).
\end{equation}
If a mass-action system $(G, \bk)$ satisfies the system \eqref{eq:xxx}, then $(G, \bk)$ is called a \defi{realization} of the system \eqref{eq:xxx}.
\end{definition}

\begin{definition}
\label{def:dyn_equ}

Let $(G,\bk)$ and $(G', \bk')$ be two mass-action systems. Then $(G,\bk)$ is said to be \textbf{dynamically equivalent} to $(G',\bk')$, denoted by $(G,\bk) \sim (G',\bk')$, if 
for every $\bx\in\mathbb{R}^n_{>0}$
\begin{equation} \notag
\displaystyle\sum_{\by \rightarrow \tilde{\by} \in E} k_{\by \rightarrow \tilde{\by}}{\bx}^{\by}(\tilde{\by} - \by) = \displaystyle\sum_{\by' \rightarrow \tilde{\by}' \in E'} k'_{\by' \rightarrow \tilde{\by}'}{\bx}^{\by'}(\tilde{\by}'- \by').
\end{equation} 
In particular, $(G',\bk')$ is a realization of $(G,\bk)$ if $(G,\bk) \sim (G',\bk')$.
\end{definition}

\begin{figure}[!ht]
    \centering
    \begin{subfigure}[t]{0.5\textwidth}
        \centering
        \includegraphics[width = 0.47\textwidth]{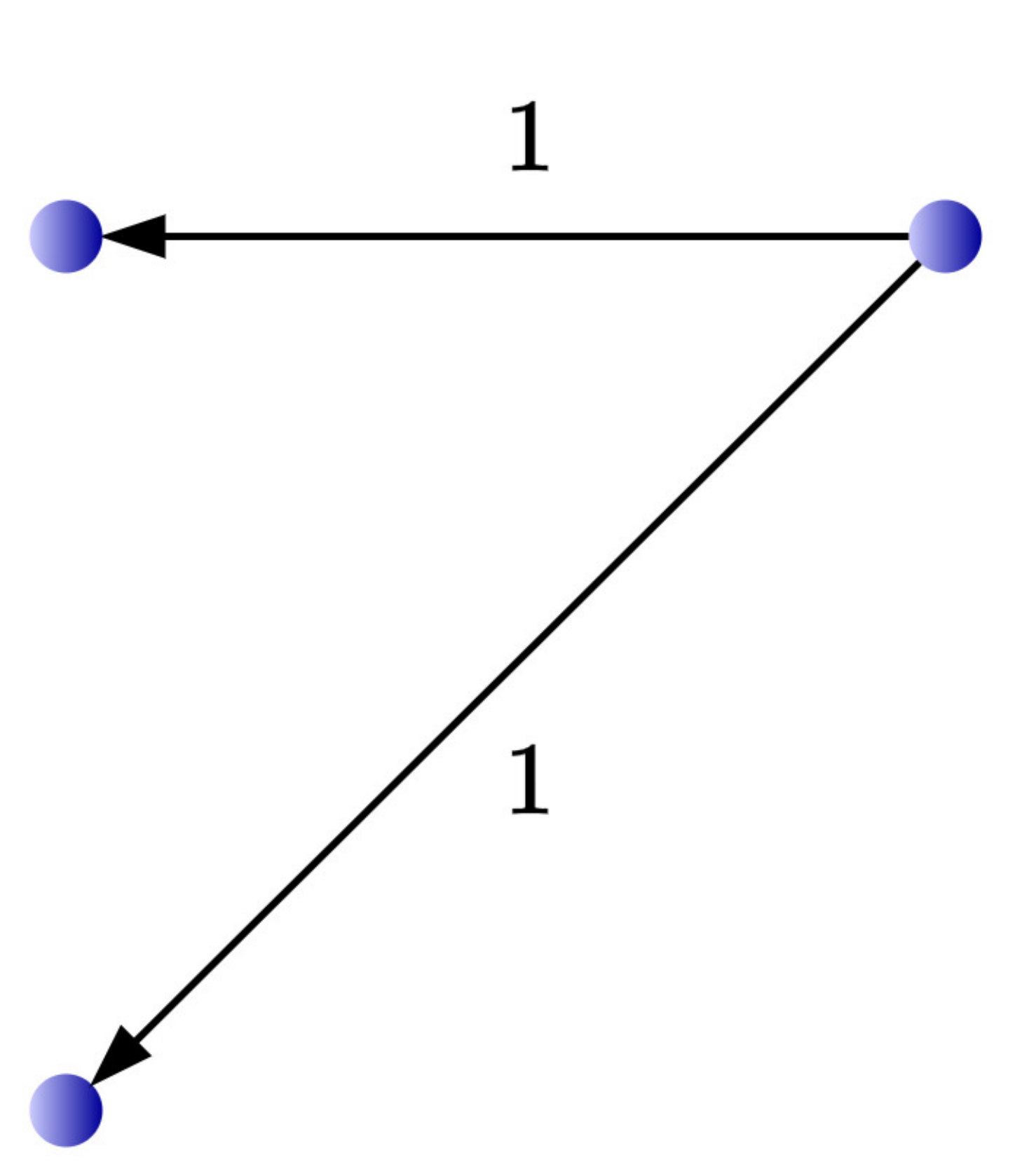}
       \caption{}
        \label{fig:DynamicalEquivalence1}
    \end{subfigure}%
    \begin{subfigure}[t]{0.5\textwidth}
        \centering
        \includegraphics[width = 0.5\textwidth]{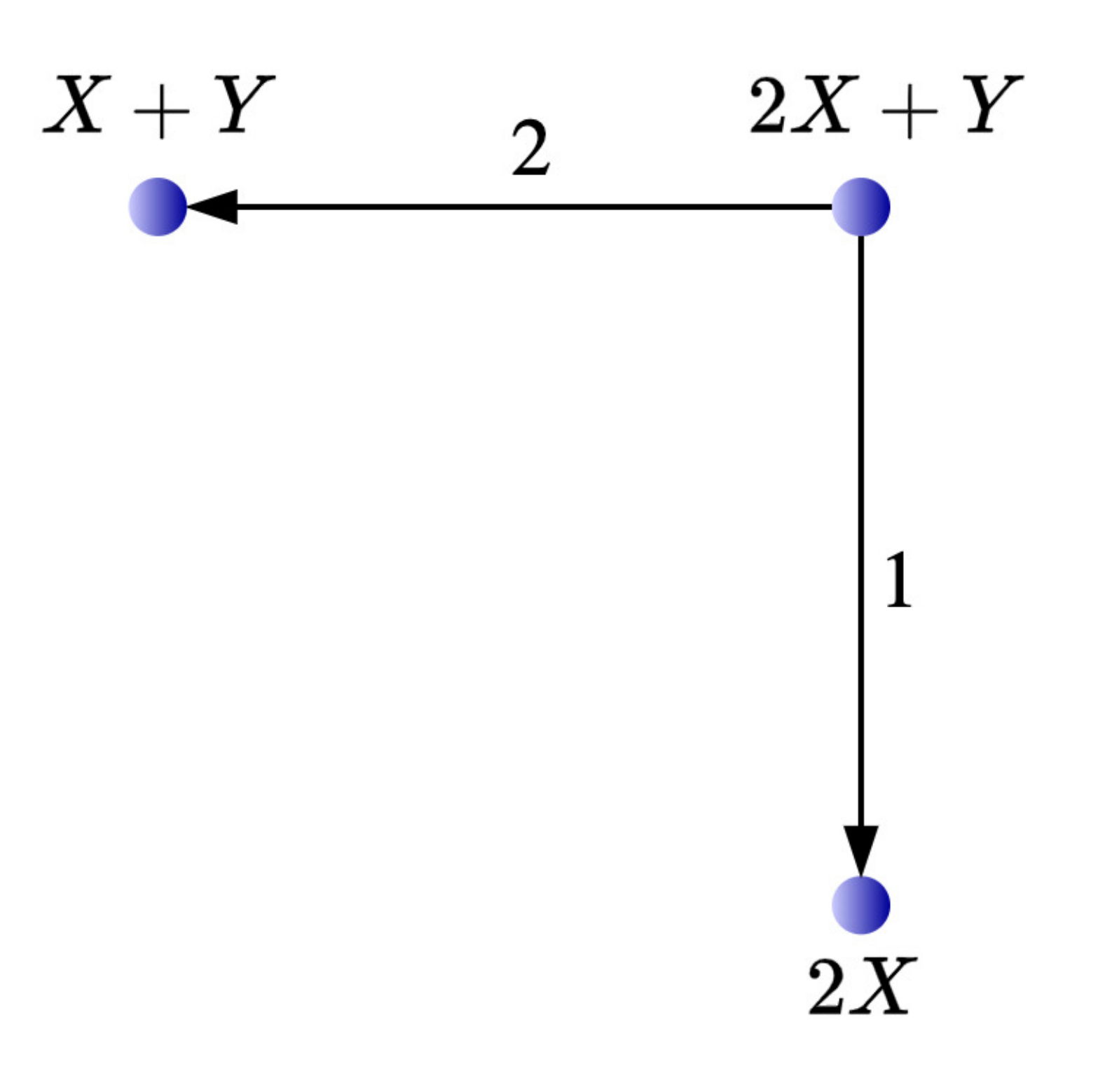}
       \caption{}
        \label{fig:DynamicalEquivalence2}
    \end{subfigure}%
    \caption{The two reaction networks are dynamically equivalent to each other and are governed by the same dynamical equation. We decompose the diagonal reaction vector from (a) into orthogonal components to get (b).}
\end{figure}

\begin{definition}
Let $G$ and $G'$ be two E-graphs. Then the dynamics of $G$ is said to be \textbf{included} within the dynamics of $G'$\footnote{
For simplicity of notation, we sometimes say that $G$ is included in $G'$ when $G\sqsubseteq G'$.
}, denoted by $G \sqsubseteq G'$, if for any $\bk\in\mathbb{R}^{|E|}_{>0}$, there exists $\bk'\in\mathbb{R}^{|E'|}_{>0}$ such that $(G,\bk) \sim (G',\bk')$.
\end{definition}

We are now ready to pose the central question of this paper precisely:

\medskip

\textbf{Question:} Given an E-graph $G=(V,E)$, what are the necessary and sufficient conditions on $G$ such that there exists an E-graph $G'\subseteq_{wr} G_c$ and $G \sqsubseteq G'$?

\medskip

It is important to note that this question can be translated into a nonlinear feasibility problem (see Appendix \ref{sec:p1} for details). 
Moreover, for a given mass-action system $(G, \bk)$, several algorithms have been proposed to check if a weakly reversible realization exists. These include mixed-integer programming methods \cite{johnston2012linear, szederkenyiactual} and linear programming-based procedures \cite{rudan2014polynomial}.
However, solving this feasibility problem can be time-consuming, particularly because it requires identifying a weakly reversible E-graph $G'$ that satisfies all reaction rate vectors $\bk\in\mathbb{R}^{|E|}_{>0}$. Additionally, the number of potential weakly reversible E-graphs grows exponentially with the number of vertices in the graph.

To enhance efficiency, we propose an alternative approach by establishing necessary and sufficient conditions based on the geometric properties of the E-graph $G$. 
In Section \ref{sec:necessary}, we show that being an \emph{endotactic} E-graph is a necessary condition. In Section \ref{sec:sufficient}, we present sufficient conditions, derived from the Newton polytope of the E-graph, that ensure the existence of a weakly reversible realization for all choices of the reaction rate vector.

\section{Necessary Conditions for an E-Graph's Dynamics to be Included in the Dynamics of a Weakly Reversible E-graph}
\label{sec:necessary}

In this section, we establish necessary conditions on the E-graph $G$ such that $(G, \bk)$ always admits a weakly reversible realization, regardless of the choice of the reaction rate vector $\bk$. More precisely, Theorem \ref{thm:necessary_wr} states that the E-graph $G$ must be endotactic.

We start by defining the \emph{net reaction vector} and the \emph{E-graph corresponding to the net reaction vectors of a mass-action system}.

\begin{definition}
Let $(G, \bk)$ be a mass-action system. For every vertex $\by\in V$, the \defi{net reaction vector} associated with $\by$ is defined as follows:
\begin{equation} \notag
\bw_{\by} = \displaystyle\sum_{\by\to\by'\in E} k_{\by\to\by'}(\by'-\by).   
\end{equation}
\end{definition}

\begin{definition}

Let $(G,\bk)$ be a mass-action system. 
The \defi{E-graph corresponding to the net reaction vectors of $(G,\bk)$}, denoted by $G_{\bW(\bk)}$, is defined as follows:
\begin{enumerate}[label=(\roman*)]
\item All source vertices of $G_{\bW(\bk)}$ correspond to the source vertices of $G$.
    
\item For every source vertex $\by$ of $G_{\bW(\bk)}$,  there exists a corresponding target vertex $\hat{\by}$ and an edge $\by \to \hat{\by} \in G_{\bW(\bk)}$ such that
\[
\hat{\by} - \by = \bw_{\by},
\]    
where $\bw_{\by}$ is the net reaction vector associated with $\by$ of $G$. 
\end{enumerate}
\end{definition}

\begin{lemma} 
\label{lem:dyn_wr_fixed}

Let $(G, \bk)$ and $(G', \bk')$ be two mass-action systems. 
Suppose $G_{\bW(\bk)}$ is the E-graph corresponding to the net reaction vectors of $(G, \bk)$.
If $G'$ is weakly reversible and $(G, \bk) \sim (G', \bk')$, then $G_{\bW(\bk)}$ is endotactic.
\end{lemma}

\begin{proof}

We prove this lemma by contradiction. Assume that $G_{\bW(\bk)}$ is not endotactic. 
From Definition \ref{def:endotactic}, this implies that there exists a reaction $\by_0 \to \hat{\by} \in E_{\bW(\bk)}$ and $\bv \in \mathbb{R}^n$ such that 
\begin{equation} \label{eq:v*y'-y<0}
\bv \cdot (\hat{\by} - \by_0) < 0.
\end{equation}
Moreover, for every $\by \to \tilde{\by} \in E_{\bW(\bk)}$ satisfying $\bv \cdot (\tilde{\by} - \by) > 0$,
\begin{equation} \label{eq:v*ty>=v*y}
\bv \cdot \by \geq \bv \cdot \by_0.
\end{equation}

Let $v_{0} = \min\limits_{\by \in V'_{s}} \{ \bv \cdot \by \}$, consider a set of source vertices in $G'$ as follows:
\begin{equation} \label{eq:v_min}
V_{\rm{min}} = \big\{ \by \in V'_{s} \ \big| \ \bv \cdot \by = v_{0} \big\}.
\end{equation}
From $G_{\bW(\bk)} \sim (G', \bk')$, then $G_{\bW(\bk)} = G'_{\bW(\bk')}$ and thus $\by_0 \to \hat{\by} \in E'_{\bW(\bk')}$.
Since $G'$ is weakly reversible, this implies that $v_0 < \bv \cdot \by_0$ and $V_{min} \neq \emptyset$.
We now claim that there exists one vertex $\by_1 \in V_{min}$ and $\by_1 \to \by'_1 \in E'$, such that
\[
\bv \cdot (\by'_1 - \by_1) < 0.
\]
Suppose not, from \eqref{eq:v*y'-y<0} and $G'$ being weakly reversible, there exists a vertex $\tilde{\by} \in V_{min}$ and
\[
\bv \cdot (\tilde{\by}'- \tilde{\by}) > 0
\ \text{ with } \
\tilde{\by} \to \tilde{\by}' \in E'_{\bW(\bk')} = E_{\bW(\bk)}.
\]
This contradicts with $\by_0 \to \hat{\by} \in E_{\bW(\bk)}$ and \eqref{eq:v*ty>=v*y} in $G_{\bW(\bk)}$, and thus we prove the claim.
Note that this claim implies that $G'$ is not endotactic. However, every weakly reversible E-graph is endotactic, which gives a contradiction. Therefore, $G_{\bW(\bk)}$ is endotactic.
\end{proof}

\begin{lemma}
\label{lem:dyn_wr_all}

Let $G = (V, E)$ be an E-graph. For every $\bk \in \mathbb{R}^{|E|}_{>0}$, there exists a weakly reversible E-graph $G'$ and $\bk' \in \mathbb{R}^{|E'|}_{>0}$ such that $(G, \bk) \sim (G', \bk')$. Then $G$ is endotactic.
\end{lemma}

\begin{proof}

We prove this lemma by contradiction. Assume that $G$ is not endotactic. 
From Definition \ref{def:endotactic}, this shows that there exists a reaction $\by \to \by' \in E$ and $\bv \in \mathbb{R}^n$ such that 
\begin{equation} \label{eq:v*y'-y<0_2}
\bv \cdot (\by'-\by) < 0.
\end{equation}
Further, for every $\tilde{\by} \to \tilde{\by}' \in E$, it  satisfies 
\begin{equation} \label{eq:v*ty>=v*y_2}
\bv \cdot (\tilde{\by}'- \tilde{\by}) \leq 0
\ \text{ or } \
\bv \cdot \tilde{\by} \geq \bv \cdot \by.
\end{equation}

Consider $G_{\bW(\bk)}$ as the E-graph corresponding to the net reaction vectors of $(G,\bk)$. 
From \eqref{eq:v*y'-y<0_2}, there exists sufficiently large constant $k_{\by \to \by'}$ such that
\[
\bv \cdot (\hat{\by} - \by) < 0
\ \text{ with } \ 
\by \to \hat{\by} \in E_{\bW(\bk)}.
\]
Moreover, \eqref{eq:v*ty>=v*y_2} implies that for any vertex $\tilde{\by} \in G_{\bW(\bk)}$ satisfying $\bv \cdot \tilde{\by} < \bv \cdot \by$, we have
\[
\bv \cdot (\tilde{\by}'- \tilde{\by}) \leq 0
\ \text{ with } \
\tilde{\by} \to \tilde{\by}' \in E_{\bW(\bk)}.
\]
This indicates $G_{\bW(\bk)}$ is not endotactic under such $\bk \in \mathbb{R}^{|E|}_{>0}$.

On the other hand, from the assumption there exists a weakly reversible E-graph $G'$ and $\bk' \in \mathbb{R}^{|E'|}_{>0}$ such that $(G, \bk) \sim (G', \bk')$ for every $\bk \in \mathbb{R}^{|E|}_{>0}$. Using Lemma \ref{lem:dyn_wr_fixed}, we derive $G_{\bW(\bk)}$ is endotactic which gives a contradiction. Therefore, $G$ is endotactic.
\end{proof}

As a consequence of Lemma \ref{lem:dyn_wr_all}, we get the main result of this section stated below.

\begin{theorem}
\label{thm:necessary_wr}

Let $G = (V, E)$ and $G' = (V', E')$ be two E-graphs. If $G'$ is weakly reversible and $G \sqsubseteq G'$, then $G$ is endotactic. 
Therefore, $G$ being endotactic is a necessary condition for its dynamics to be included in the dynamics of a weakly reversible E-graph.
\end{theorem}

\begin{proof}

Since $G\sqsubseteq G'$, for any $\bk \in \mathbb{R}^{|E|}_{>0}$ there exists $\bk' \in \mathbb{R}^{|E'|}_{>0}$, such that $(G,\bk) \sim (G',\bk')$.
Using Lemma \ref{lem:dyn_wr_all}, we conclude that $G$ is endotactic.
\end{proof}

\section{Sufficient Conditions for an E-Graph's Dynamics to be Included in the Dynamics of a Weakly Reversible E-graph}
\label{sec:sufficient}

This section aims to establish sufficient conditions on the E-graph $G$ to ensure that $G$ can be included in a weakly reversible E-graph.
According to Theorem~\ref{thm:necessary_wr}, $G$ must be an endotactic network to start with.
We divide our analysis into two cases: (i) Strongly endotactic 2D networks with two-dimensional stoichiometric subspaces, and (ii) Endotactic networks in arbitrary dimensions.

We introduce the following notation, which will be used throughout this section.

Consider an E-graph $G = (V, E)$ and a vertex $\by \in V$. Then
\begin{enumerate}[label=(\roman*)]
\item $Cone_{G}(\by) := \{\rm{Cone}(\by'-\by) \,|\, \by\to\by'\in E \} = \left\{\displaystyle\sum_{\by\to\by'\in E} \lambda_{\by\to\by'} (\by'-\by) \mid\, \lambda_{\by\to\by'} >0\right\}$.

\item $RelInt[Cone_{G}(\by)]$: the relative interior of $Cone_{G}(\by)$.

\item $\partial Cone_{G}(\by)$: the boundary of $Cone_{G}(\by)$.
\end{enumerate}

\subsection{Strongly Endotactic 2D Networks with Two-Dimensional Stoichiometric Subspaces}
\label{sec:strongly_endo_two}

In this section, we prove the dynamics of a strongly endotactic 2D E-graph with a two-dimensional stoichiometric subspace, along with certain additional structural properties, can be realized by the dynamics of a weakly reversible network consisting of a single linkage class. 

\begin{theorem} 
\label{thm:our_theorem}

Let $G =(V, E)$ be a strongly endotactic 2D E-graph with a two-dimensional stoichiometric subspace. 
Assume that all source vertices lie on the boundary of $\New (G)$.  
Then there exists a weakly reversible E-graph $G'$ such that $G'$ has a single linkage class and $G \sqsubseteq G'$
\end{theorem}

\begin{proof}

Let $\{ \by_1, \ldots, \by_m \}$ denote source vertices in $G = (V, E)$ and all of them lie on the boundary of $\New (G)$ from the setting.
Since two-dimensional E-graph $G$ has a two-dimensional stoichiometric subspace, we assume that the source vertices $\by_1, \ldots, \by_m$ are in clockwise order over $\New (G)$.
Further, we divide source vertices $\{ \by_1, \ldots, \by_m \}$ into two categories: 
\begin{enumerate}[label=(\roman*)]
\item For $1 \leq i \leq m$, $\by_i$ is called a \emph{corner vertex} if
\begin{equation} \label{def:courner}
\by_{i} - \by_{i-1}
\ \text{ is linearly independent to }
\by_{i+1} - \by_{i}.
\end{equation}
We let $\{ \by_{c_1}, \ldots, \by_{c_k} \}$ denote all corner vertices in $G$ under the clockwise order. 

\item For $1 \leq i \leq m$, $\by_i$ is called a \emph{side vertex} if there exists $k \in \R$, such that 
\begin{equation} \label{def:side}
\by_{i} - \by_{i-1} 
= k ( \by_{i+1} - \by_{i} ).
\end{equation}
We let $\{ \by_{s_1}, \ldots, \by_{s_l} \}$ denote all side vertices in $G$ under the clockwise order.
\end{enumerate} 
Here we abuse the notation by letting $\by_{m+1}$ denote $\by_{1}$ and letting $\by_{c_{k+1}}$ denote $\by_{c_1}$.

Since all source vertices are on the boundary of $\New (G)$, every side vertex must lie on the line between two adjacent corner vertices. Thus, for any $1 \leq i \leq l$, there exists two adjacent corner vertices $\by_{c_j}, \by_{c_{j+1}}$ such that
\begin{equation} \label{side_corner}
\by_{s_i} - \by_{c_j} = \nu ( \by_{c_{j+1}} - \by_{c_j} )
\ \text{ with } \
0 < \nu < 1.
\end{equation}
Hence, we let $\{ \by_{s_{q(j)+1}}, \ldots, \by_{s_{q(j+1)}} \}$ denote all side vertices between $\by_{c_j}$ and $\by_{c_{j+1}}$.
Furthermore, since $G =(V, E)$ is strongly endotactic, then for any $\by \to \by' \in E$,
\begin{equation} \label{relint_partial}
\by' - \by \in relint[Cone_{G} (\by)]
\ \text{ or } \ 
\by' - \by \in \partial Cone_{G} (\by).
\end{equation}

Now we are ready to construct an E-graph $G' = (V', E')$ whose vertices are all source vertices in $G$, that is,
\[
V' = \{ \by_1, \ldots, \by_m \} 
= \{ \by_{c_1}, \ldots, \by_{c_k} \}
\cup
\{ \by_{s_1}, \ldots, \by_{s_l} \}.
\]
In the following, we build the edges $E'$ in $G'$ and we will show that the dynamics of $G$ is included within the dynamics of $G'$ and $G'$ it is weakly reversible and has a single linkage class.

\smallskip

\textbf{Step 1:}
We start with the corner vertices $\{ \by_{c_1}, \ldots, \by_{c_k} \}$.
Without loss of generality, we first consider the corner vertex $\by_{c_1}$ and all edges from $\by_{c_1}$. From \eqref{relint_partial}, there are two possibilities for the edges $\by_{c_1} \to \by \in E$, that is,
\begin{equation} \notag
\by - \by_{c_1} \in relint[Cone_{G} (\by_{c_1})]
\ \text{ or } \ 
\by - \by_{c_1} \in \partial Cone_{G} (\by_{c_1}).
\end{equation}

$(a)$
If the edge $\by_{c_1} \to \by \in E$ and $\by - \by_{c_1} \in relint[Cone_{G} (\by_{c_1})]$, then we add two edges in $E'$ as follows:
\[
\by_{c_1} \to \by_{c_k} \in E'
\ \text{ and } \
\by_{c_1} \to \by_{c_2} \in E'.
\]
From \eqref{def:courner}, \eqref{side_corner} and the assumption that $G$ has a two-dimensional stoichiometric subspace $S$, we obtain that
\[
S = \spn \{ \by_{c_k} - \by_{c_1}, \ \by_{c_2} - \by_{c_1} \}.
\]
Together with $\by_{c_k}, \by_{c_2}$ are two adjacent corner vertices to $\by_{c_1}$, there exist two positive constants $\alpha, \beta > 0$ such that
\begin{equation} \label{realization:int}
\by - \by_{c_1} = \alpha ( \by_{c_k} - \by_{c_1} ) + \beta ( \by_{c_2} - \by_{c_1} ).
\end{equation}

$(b)$
Otherwise, the edge $\by_{c_1} \to \by \in E$ and $\by - \by_{c_1} \in \partial Cone_{G} (\by_{c_1})$. 
Since $\by_{c_k}, \by_{c_2}$ are two adjacent corner vertices to $\by_{c_1}$, there exist a positive constant $\gamma > 0$ such that
\begin{equation} \label{realization:partial}
\by - \by_{c_1} = \gamma ( \by_{c_k} - \by_{c_1} )
\ \text{ or } \
\by - \by_{c_1} = \gamma ( \by_{c_2} - \by_{c_1} ),
\end{equation}
and thus we add the edge $\by_{c_1} \to \by_{c_k} \in E'$ or $\by_{c_1} \to \by_{c_2} \in E'$, respectively.

\smallskip

\textbf{Step 2:}
Applying the operation in the cases $(a)$ and $(b)$ on all edges $\by_{c_1} \to \by \in E$, we have added one or both edges as follows:
\begin{equation} \label{two_edges}
\by_{c_1} \to \by_{c_k} \in E'
\ \text{ or } \
\by_{c_1} \to \by_{c_2} \in E'.
\end{equation}

First, suppose only one edge in \eqref{two_edges} is included in $E'$. Assume that $\by_{c_1} \to \by_{c_2} \in E'$.
Recall that $\{ \by_{s_{q(1)+1}}, \ldots, \by_{s_{q(2)}} \}$ denote all side vertices between $\by_{c_1}$ and $\by_{c_2}$, which satisfy  
\begin{equation} \label{side_corner_12}
\by_{s_i} - \by_{c_1} = \nu_i ( \by_{c_{2}} - \by_{c_j} )
\ \text{ with } \
0 < \nu_i < 1
\ \text{ for } \
q(1)+1 \leq i \leq q (2).
\end{equation}
Then we add extra edges from $\by_{c_1}$ to all these side vertices in $E'$, that is, 
\begin{equation} \label{side_edges_12_E'}
\by_{c_1} \to \by_{i} \in E'
\ \text{ for } \
q(1)+1 \leq i \leq q (2).
\end{equation}
For any positive constants $c > 0$, there exists sufficiently small weights $\{ \epsilon_{i}  \}^{q (2)}_{i= q(1)+1}$, such that
\[
c - \sum\limits^{q (2)}_{i= q(1)+1} \epsilon_i \nu_i > 0
\ \text{ with } \
\epsilon_{q(1)+1}, \ldots \epsilon_{q (2)} > 0.
\]
Thus we obtain that
\begin{equation} \notag
\begin{split}
c ( \by_{c_2} - \by_{c_1} ) 
& = \big( c - \sum\limits^{q (2)}_{i= q(1)+1} \epsilon_i \nu_i \big) ( \by_{c_2} - \by_{c_1} ) 
+ \sum\limits^{q (2)}_{i= q(1)+1} \epsilon_i \nu_i ( \by_{c_2} - \by_{c_1} ).
\end{split}
\end{equation}
Therefore, if only one edge in \eqref{two_edges} is included in $E'$, the dynamics of \eqref{two_edges} is included within the dynamics of \eqref{side_edges_12_E'}.

Second, suppose both edges in \eqref{two_edges} are included in $E'$. Then we add edges from $\by_{c_1}$ to all other vertices in $E'$, that is, 
\begin{equation} \label{all_edges_E'}
\by_{c_1} \to \by_{i} \in E'
\ \text{ for } \
1 \leq i \leq m, \
i \neq c_1.
\end{equation}
Since $\{ \by_{1}, \ldots, \by_{m} \}$ lie on the boundary of $\New (G)$, then for any $1 \leq i \leq m$ 
\[
\by_{i} - \by_{c_1} \in Cone_{G} (\by_{c_1}).
\]
and thus there exist two non-negative constants $\alpha_i, \beta_i \geq 0$, such that
\begin{equation} \label{realization:corner}
\by_{i} - \by_{c_1} = \alpha_i ( \by_{c_k} - \by_{c_1} ) + \beta_i ( \by_{c_2} - \by_{c_1} ).
\end{equation}
Given any positive constants $c_1 , c_2 > 0$, there exists sufficiently small weights $\varepsilon_i$ with $1 \leq i \leq m$, such that
\[
c_1 - \sum\limits^{m}_{i=1} w_i \alpha_i > 0,
\ \
c_2 - \sum\limits^{m}_{i=1} w_i \alpha_i > 0.
\]
Therefore, we obtain that
\begin{equation} \notag
\begin{split}
& \ \ \ \ c_1 ( \by_{c_k} - \by_{c_1} ) + c_{2} ( \by_{c_2} - \by_{c_1} ) 
\\& = \big( c_1 - \sum\limits^{m}_{i=1} w_i \alpha_i \big) ( \by_{c_k} - \by_{c_1} ) + \big( c_2 - \sum\limits^{m}_{i=1} w_i \alpha_i \big) ( \by_{c_2} - \by_{c_1} ) 
\\& \qquad + \sum\limits^{m}_{i=1} w_i \big( \alpha_i ( \by_{c_k} - \by_{c_1} ) + \beta_i ( \by_{c_2} - \by_{c_1} ) \big)
\\& = \big( c_1 - \sum\limits^{m}_{i=1} w_i \alpha_i \big) ( \by_{c_k} - \by_{c_1} ) + \big( c_2 - \sum\limits^{m}_{i=1} w_i \alpha_i \big) ( \by_{c_2} - \by_{c_1} ) + \sum\limits^{m}_{i=1} w_i (\by_{i} - \by_{c_1}).
\end{split}
\end{equation}
Therefore, if both edges in \eqref{two_edges} are included in $E'$, the dynamics of \eqref{two_edges} is included within the dynamics of \eqref{all_edges_E'}.

Analogously, following steps 1 and 2 we work on the rest of corner vertices $\by_{c_2}, \ldots, \by_{c_k}$, and add edges from these corner vertices.

\smallskip

\textbf{Step 3:}
Now we deal with the side vertices $\{ \by_{s_1}, \ldots, \by_{s_l} \}$.
Without loss of generality, we start with the side vertex $\by_{s_1}$ and all edges from $\by_{s_1}$. 
Assume that $\by_{s_1}$ is between two adjacent corner vertices $\by_{c_j}, \by_{c_{j+1}}$, thus we have
\begin{equation} \label{side_corner_1}
\by_{s_1} - \by_{c_j} = \nu ( \by_{c_{j+1}} - \by_{c_j} )
\ \text{ with } \
0 < \nu < 1.
\end{equation}
From \eqref{relint_partial}, there are two possibilities for the edges $\by_{s_1} \to \by \in E$, that is,
\begin{equation} \notag
\by - \by_{s_1} \in relint[Cone_{G} (\by_{s_1})]
\ \text{ or } \ 
\by - \by_{s_1} \in \partial Cone_{G} (\by_{s_1}).
\end{equation}

$(a)$
If the edge $\by_{s_1} \to \by \in E$ and $\by - \by_{s_1} \in relint[Cone_{G} (\by_{s_1})]$, then we add three edges in $E'$ as follows:
\begin{equation} \label{three_edges}
\by_{s_1} \to \by_{c_j} \in E', \ \
\by_{s_1} \to \by_{c_{j+1}} \in E'
\ \text{ and } \
\by_{s_1} \to \by_{c_{j+2}} \in E'.
\end{equation}
Since $\by_{c_j}, \by_{c_{j+1}}$ are two adjacent corner vertices to $\by_{s_1}$, together with all source vertices are on the boundary of $\New (G)$, then the reaction vector $\by - \by_{s_1}$ must be in the non-negative span of 
\[
\{ \by_{c_j} - \by_{s_1}, \by_{c_{j+2}} - \by_{s_1} \}
\ \text{ or } \
\{ \by_{c_{j+1}} - \by_{s_1}, \by_{c_{j+2}} - \by_{s_1} \}.
\]
Thus we assume that there exist two positive constants $\alpha, \beta \geq 0$ such that
\begin{equation} \label{realization:int_side}
\by - \by_{s_1} = \alpha ( \by_{c_j} - \by_{s_1} ) + \beta ( \by_{c_{j+2}} - \by_{s_1} ).
\end{equation}
Recall from \eqref{side_corner_1}, we compute that
\begin{equation} \label{side_corner_1_j}
(1 - \nu) ( \by_{c_j} - \by_{s_1} ) + \nu ( \by_{c_{j+1}} - \by_{s_1} ) = 0
\ \text{ with } \
0 < \nu < 1.
\end{equation}
Then for any positive constant $c > 0$, together with \eqref{realization:int_side}, we obtain that
\begin{equation} \notag
\begin{split}
c ( \by - \by_{s_1} ) 
& = c \alpha ( \by_{c_j} - \by_{s_1} ) + c \beta ( \by_{c_{j+2}} - \by_{s_1} )
\\& = \big( c \alpha + (1 - \nu) \big) ( \by_{c_j} - \by_{s_1} ) + c \beta ( \by_{c_{j+2}} - \by_{s_1} ) + \nu ( \by_{c_{j+1}} - \by_{s_1} ).
\end{split}
\end{equation}
Therefore, the dynamics of $\by_{s_1} \to \by$ is included within the dynamics of \eqref{three_edges}.

$(b)$
Otherwise, the edge $\by_{s_1} \to \by \in E$ and $\by - \by_{s_1} \in \partial Cone_{G} (\by_{s_1})$. Then we add two edges in $E'$ as follows:
\begin{equation} \label{two_edges_side}
\by_{s_1} \to \by_{c_j} \in E'
\ \text{ and } \
\by_{s_1} \to \by_{c_{j+1}} \in E'.
\end{equation}
Recall that $\by_{c_j}, \by_{c_{j+1}}$ are two adjacent corner vertices to $\by_{s_1}$, there exist a positive constant $\gamma > 0$ such that
\begin{equation} \notag
\by - \by_{s_1} = \gamma ( \by_{c_j} - \by_{s_1} )
\ \text{ or } \
\by - \by_{s_1} = \gamma ( \by_{c_{j+1}} - \by_{s_1} ).
\end{equation}
Assume that there exist a positive constant $\alpha, \beta > 0$, such that
\begin{equation} \label{realization:partial_side}
\by - \by_{s_1} = \gamma ( \by_{c_j} - \by_{s_1} ).
\end{equation}
Then for any positive constant $c > 0$, together with \eqref{side_corner_1_j} and \eqref{realization:partial_side}, we obtain that
\begin{equation} \notag
\begin{split}
c ( \by - \by_{s_1} ) 
= c \gamma ( \by_{c_j} - \by_{s_1} )
= \big( c \gamma + (1 - \nu) \big) ( \by_{c_j} - \by_{s_1} ) + \nu ( \by_{c_{j+1}} - \by_{s_1} ).
\end{split}
\end{equation}
Therefore, the dynamics of $\by_{s_1} \to \by$ is included within the dynamics of \eqref{two_edges_side}.

Then we apply the operation in the cases $(a)$ and $(b)$ on all edges $\by_{s_1} \to \by \in E$, and the graph $E'$ must include two edges as follows:
\begin{equation} \label{two_edges_side_E'}
\by_{s_1} \to \by_{c_j} \in E'
\ \text{ and } \
\by_{s_1} \to \by_{c_{j+1}} \in E'.
\end{equation}
Analogously, following step 3 we work on the rest of the side vertices $\by_{s_2}, \ldots, \by_{s_l}$, and add edges from these side vertices.

\smallskip

\textbf{Step 4:}
Now we complete the construction of the E-graph $G' = (V', E')$, and we can check that the dynamics of $G$ is included within the dynamics of $G'$. It remains to show that $G' = (V', E')$ is weakly reversible and has a single linkage class.  

We claim that for any distinct corner vertices $\by_{c_i}, \by_{c_j}$ with $i \neq j$, there exists a path in $E'$ from $\by_{c_i}$ to $\by_{c_j}$.
Without loss of generality, we start with the corner vertex $\by_{c_1}$.
From step 1, we have added one or both edges from $\by_{c_1}$ to its adjacent corner vertices $\by_{c_2}, \by_{c_k}$ in $E'$, that is,
\[
\by_{c_1} \to \by_{c_k}
\ \text{ or } \
\by_{c_1} \to \by_{c_2}.
\]
Assume that $\by_{c_1} \to \by_{c_2} \in E'$, then we focus on the corner vertex $\by_{c_2}$ and consider the edge from $\by_{c_2}$.
Following this pattern, there are two possibilities:
\begin{enumerate}[label=(\roman*)]
\item For any $1 \leq i \leq k$, $\by_{c_i} \to \by_{c_{i+1}} \in E'$. Here $\by_{c_{k+1}}$ denotes $\by_{c_1}$.

\item There exists some $2 \leq i \leq k$, such that $\by_{c_i} \to \by_{c_{i+1}} \notin E'$.
\end{enumerate}

In the first case, there exists a cycle in $E'$ as follows:
\[
\by_{c_1} \to \by_{c_2} \to \ldots \to \by_{c_{k}} \to \by_{c_1}.
\]
Thus there is a path between arbitrary two distinct corner vertices in $E'$, and we prove the claim.

In the second case, assume that $i$ is the first index such that 
\[
\by_{c_i} \to \by_{c_{i+1}} \notin E'
\ \text{ with } \
2 \leq i \leq k.
\]
From step 2, together with \eqref{side_edges_12_E'}, there is an edge $\by_{c_i} \to \by_{c_{i-1}} \in E'$ and other edges from $\by_{c_i}$ follows 
\begin{equation} \notag
\by_{c_i} \to \by_{j} \in E'
\ \text{ for } \
q(i-1)+1 \leq j \leq q (i),
\end{equation}
where $\by_{q(i-1)+1}, \ldots, \by_{q (i)}$ are side vertices between corner vertices $\by_{c_i}$ and $\by_{c_{i-1}}$.

Now consider the corner vertex $\by_{c_{i-1}}$.
From the assumption, we have 
\[
\by_{c_{i-1}} \to \by_{c_{i}} \in E'
\ \text{ and } \
\by_{c_{i-2}} \to \by_{c_{i-1}} \in E'.
\]
Note that $\by_{c_1}, \ldots, \by_{c_k}$ are corner vertices from the boundary of $\New (G)$.
There exists a vector $\bv \in \R^n$ such that
\begin{equation} \notag
\bv \cdot ( \by_{c_{i-1}} - \by_{c_{i-2}} ) < 0
\ \text{ and } \
\bv \cdot ( \by_{c_{i}} - \by_{c_{i-1}} ) = 0.
\end{equation}
Since $G$ is strongly endotactic, together with $\by_{c_i} \to \by_{c_{i+1}} \notin E'$, we get
\[
\by_{c_{i-1}} \to \by_{c_{i-2}} \in E'.
\]
From step 2, $E'$ includes all edges from $\by_{c_{i-1}}$ to every other vertices, and thus 
\begin{equation} \notag
\by_{c_{i-1}} \to \by_{c_{j}} \in E'
\ \text{ with } \
1 \leq j \leq k, \ j \neq i.
\end{equation}

We proceed to the vertex $\by_{c_{i+1}}$. 
From step 1, $E'$ contains one or both edges from $\by_{c_{i+1}}$ to its adjacent corner vertices $\by_{c_{i}}, \by_{c_{i+2}}$, that is,
\[
\by_{c_{i+1}} \to \by_{c_{i}}
\ \text{ or } \
\by_{c_{i+1}} \to \by_{c_{i+2}}.
\]
If $\by_{c_{i+1}} \to \by_{c_{i}} \in E'$, then we obtain an cycle 
\[
\by_{c_{i-1}} \to \by_{c_{i+1}} \to \by_{c_{i}} \to \by_{c_{i-1}} \in E'.
\]
Otherwise, $\by_{c_{i+1}} \to \by_{c_{i+2}} \in E'$ and we have a new path
\[
\by_{c_{i-1}} \to \by_{c_{i+1}} \to \by_{c_{i+2}} \in E'.
\]
In both cases, we can proceed to the vertex $\by_{c_{i+2}}$. Analogously, we iterate the above argument on the rest of the corner vertices and thus conclude the claim.

\smallskip

\textbf{Step 5:}
Next we claim that for any side vertex $\by_{s_i}$ with $1 \leq i \leq l$, there exists corner vertices $\by_{c_a}, \by_{c_b}$ such that
\[
\by_{c_a} \to \by_{s_i} \in E'
\ \text{ and } \
\by_{s_i} \to \by_{c_b} \in E'.
\]

Without loss of generality, we consider the side vertex $\by_{s_1}$. Assume that $\by_{s_1}$ is between two adjacent corner vertices $\by_{c_j}, \by_{c_{j+1}}$, together with \eqref{two_edges_side_E'} in step 3, the graph $E'$ includes two edges as follows:
\begin{equation} \notag
\by_{s_1} \to \by_{c_j} \in E'
\ \text{ and } \
\by_{s_1} \to \by_{c_{j+1}} \in E'.
\end{equation}
Therefore, we find edges from $\by_{s_1}$ to corner vertices.
Similarly, from step 3 there are edges from all side vertices to their adjacent corner vertices in $E'$.

On the other hand, from step 4 we obtain that either there exists a cycle in $E'$ given by
\[
\by_{c_1} \to \by_{c_2} \to \ldots \to \by_{c_{k}} \to \by_{c_1}.
\]
or there exists a corner vertex $\by_{c_a}$ such that 
\begin{equation} \notag
\by_{c_a} \to \by_{i} \in E'
\ \text{ with } \
1 \leq i \leq m, \ i \neq c_a.
\end{equation}
Consider the side vertex $\by_{s_1}$. 
In the first case, since $\by_{c_j} \to \by_{c_{j+1}} \in E'$, there is an edge $\by_{c_j} \to \by_{s_1} \in E'$ from the construction in step 2.
In the second case, it is clear that $\by_{c} \to \by_{s_1} \in E'$.
Analogously, edges from corner vertices to every side vertex in $E'$ exist from the construction.
Together with the previous step, we prove this claim.

\smallskip

\textbf{Step 6:}
Finally, we prove that $G' = (V', E')$ is weakly reversible and has a single linkage class.
It suffices for us to show that for any distinct vertices $\by_{i}, \by_{j}$ with $i \neq j$, there exists a path in $E'$ from $\by_{i}$ to $\by_{j}$.

Given two distinct vertices $\by_{i}, \by_{j}$, suppose that $\by_{i}, \by_{j}$ are both corner vertices, then from the claim in step 4 there exists a path from $\by_{i}$ to $\by_{j}$.
If both of $\by_{i}, \by_{j}$ are side vertices, from the claim in step 5 there exist two corner vertices $\by_{c_a}, \by_{c_b}$ such that
\[
\by_{i} \to \by_{c_a} \in E'
\ \text{ and } \
\by_{c_b} \to \by_{j} \in E'.
\]
Then using the claim in step 4, there exists a path from $\by_{c_a}$ to $\by_{c_b}$, and thus
\[
\by_{i} \to \by_{c_a} \to \cdots \to \by_{c_b} \to \by_{j}.
\]
We omit the case when one of $\by_{i}, \by_{j}$ is a side vertex since it directly follows from the above.
\end{proof}

\begin{theorem}
\label{thm:two_strong_endo_weak}

Let $G =(V, E)$ be a strongly endotactic 2D E-graph with a two-dimensional stoichiometric subspace. Assume there exists a source vertex $\by_0$ on the boundary of $\New (G)$ such that the net reaction vector corresponding to $\by_0$ points strictly in the interior of $\New (G)$. Then there exists a weakly reversible single linkage E-graph $G'$ such that $G\sqsubseteq G'$.
\end{theorem}

\begin{proof}

From Definition \ref{def:newton}, any vertex $\by \in G$ has two possible cases: 
\begin{enumerate}[label=(\roman*)]
\item $\by$ lies on the boundary of $\New (G)$,

\item $\by$ lies in the interior of $\New (G)$.
\end{enumerate}

First, we consider the vertices on the boundary of $\New (G)$, denoted by
\begin{equation} \notag
\{ \by_0, \ldots, \by_b \}.
\end{equation}
Using Theorem~\ref{thm:our_theorem}, we obtain that the network formed by considering only the vertices on the boundary of $\New (G)$ is weakly reversible and consists of a single linkage.

Without loss of generality, we let $\by_0$ denote a vertex whose net reaction vector $\bw_{\by_0}$ points strictly into the interior of $\New (G)$. Further, we let $\by^*$ denote any arbitrary vertex in the interior of $\New (G)$. Since $\New (G)$ is the convex hull formed by the source vertices of $G$, together with the assumption that the stoichiometric subspace is two-dimensional, there exist two vertices $\by_1, \by_2$ on the boundary of $\New (G)$, such that for any vertex $\by \in V$,
\begin{equation} \notag
\by - \by_0 \in \{ \alpha_1 (\by_1 - \by_0) + \alpha_2 (\by_2 - \by_0) \ \big| \ \alpha_1, \alpha_2 \in \R_{\geq 0} \}.
\end{equation}
Furthermore, for any vector $\bv$ pointing strictly into the interior of $\New (G)$, it can be expressed as a positive combination of $\by_1, \by_2$, that is,
\begin{equation} \notag
\bv \in \{ \alpha_1 (\by_1 - \by_0) + \alpha_2 (\by_2 - \by_0) \ \big| \ \alpha_1, \alpha_2 \in \R_{> 0} \}.
\end{equation}

Note that from the assumption on $\bw_{\by_0}$, there exist positive constants $\alpha_{0, 1}, \alpha_{0, 2} > 0$, such that
\begin{equation} \notag
\bw_{\by_0} = \alpha_{0, 1} (\by_1 - \by_0) + \alpha_{0, 2} (\by_2 - \by_0).
\end{equation}
Similarly, since $\by^*$ is in the interior of $\New (G)$, there exist positive constants $\alpha^*_{1}, \alpha^*_{2} > 0$, such that the vector $\by^* - \by_0$ follows
\begin{equation} \notag
\by^* - \by_0 = \alpha^*_{1} (\by_1 - \by_0) + \alpha^*_{2} (\by_2 - \by_0).
\end{equation}
Then there exists a sufficiently small value $k$ satisfying
\begin{equation} \notag
k \alpha^*_{1} < \alpha_{0, 1}
\ \text{ and } \
k \alpha^*_{2} < \alpha_{0, 2},
\end{equation}
and thus we get
\begin{equation} \notag
\begin{split}
\bw_{\by_0} 
& = (\alpha_{0, 1} - k \alpha^*_{1}) (\by_1 - \by_0) + (\alpha_{0, 2} - k \alpha^*_{2}) (\by_2 - \by_0) + \big( k \alpha^*_{1} (\by_1 - \by_0) + k \alpha^*_{2} (\by_2 - \by_0) \big)
\\& = 
\underbrace{
(\alpha_{0, 1} - k \alpha^*_{1}) (\by_1 - \by_0) + (\alpha_{0, 2} - k \alpha^*_{2}) (\by_2 - \by_0)
}_{\tilde{\bw}_{\by_0}}
+ k (\by^* - \by_0).
\end{split}
\end{equation}

Since $\tilde{\bw}_{\by_0}$ is in the positive span of $\{ \by_1 - \by_0, \ \by_2 - \by_0 \}$, this implies that $\tilde{\bw}_{\by_0}$ points strictly into the interior of $\New (G)$.
Consider $\tilde{\bw}_{\by_0}$ as the modified net reaction vector of $\by_0$, we again use Theorem~\ref{thm:our_theorem} to obtain a realization that is weakly reversible and consists of a single linkage class (with respect to the source vertices on the boundary of $\New (G)$).

On the other hand, the vector $k (\by^* - \by_0)$ can be realized by adding an edge from $\by_0$ to $\by^*$ with the reaction rate constant $k$. Note that $\by^*$ represents any arbitrary vertex in the interior of $\New (G)$. Hence, we can iterate the above step on all vertices in the interior of $\New (G)$. This gives a realization $G' = (V', E')$ which satisfies that for any vertex $\by^*$  in the interior of $\New (G)$,
\begin{equation} \label{eq:y0-y*}
\by_0 \to \by^* \in E'.
\end{equation}
Thus, source vertices on the boundary of $\New (G)$ form a weakly reversible realization consisting of a single linkage class.

\medskip

Next, we consider the vertices in the interior of $\New (G)$.
Recall that $\New (G)$ is the convex hull formed by the source vertices of $G$.
Assume $\by^*$ is a vertex in the interior of $\New (G)$, then for any vertex $\by \in V$,
\begin{equation} \label{eq:y*-y0}
\by - \by^* \in \Big\{ \sum\limits^{b}_{i = 0} \alpha_i (\by_i - \by^*) \ \big| \ \alpha_0, \ldots, \alpha_b \in \R_{> 0} \Big\}.
\end{equation}

Consider $\bw_{\by^*}$ the net reaction vector on $\by^*$, which is the positive combination of reaction vector from $\by^*$, from \eqref{eq:y*-y0} there exist positive values $\alpha^*_0, \ldots, \alpha^*_d$, such that
\begin{equation} \label{eq2:y*-y0}
\bw_{\by^*} = \sum\limits^{b}_{i = 0} \alpha^*_i (\by_i - \by^*)
\ \text{ with } \
\alpha^*_i > 0.
\end{equation}
Thus, the net reaction vector on $\by^*$ can be realized by
\begin{equation}\notag
\by^* \xrightarrow{\alpha^*_i} \by_i
\ \text{ for } \
i = 0, \ldots, b.
\end{equation}
Analogously, we can iterate the above step on all vertices in the interior of $\New (G)$.
This gives a realization $G' = (V', E')$ which satisfies that for any vertex $\by$  in the interior of $\New (G)$,
\begin{equation} \notag
\by \to \by_i \in E'
\ \text{ for } \
i = 0, \ldots, b.
\end{equation}

Our analysis shows that $G' = (V', E')$ is weakly reversible and consists of a single linkage class such that $G \sqsubseteq G'$ and we conclude our result.
\end{proof}

\begin{theorem}[\cite{boyd2004convex}]
\label{thm:separating_hyperplane}

Let $U, V$ be two non-empty convex sets in $\mathbb{R}^n$ such that $U\cap V\neq\emptyset$. Then there exists a vector $0\neq z\in\mathbb{R}^n$ and $\zeta\in\mathbb{R}$ such that for every $x\in U$ and $y\in V$,
\[
\langle y, z\rangle \leq \zeta \leq \langle x, z\rangle.
\]
\end{theorem}

\begin{lemma}
\label{lem:newt_G_G'}

Let $G$ and $G'$ be two E-graphs. 
Suppose that $G'$ is weakly reversible and has a single linkage class with $G \sqsubseteq G'$, then 
\begin{equation}\notag
\label{eq:newt_G_G'}
S_{G} = S_{G'}
\ \text{ and } \
\New (G) = \New (G').
\end{equation}
\end{lemma}

\begin{proof}

First, we prove that $S_{G} = S_{G'}$.
From $G \sqsubseteq G'$, given any $\bk \in \mathbb{R}^{|E|}_{>0}$ there exists $\bk' \in \mathbb{R}^{|E'|}_{>0}$ such that for every $\by_0 \in V \cup V'$,
\begin{equation} \label{eq:w_y0=w'_y0}
\underbrace{
\sum\limits_{\by_0 \to \by \in E}
k_{\by_0 \to \by'} ( \by' - \by_0 )
}_{\bw_{\by_0}: \text{ net reaction vector on $\by_0$ in $G$}} \
= \ \underbrace{
\sum\limits_{\by_0 \to \by' \in E'} k'_{\by_0 \to \by'} (\by'- \by_0)
}_{\bw'_{\by_0}: \text{ net reaction vector on $\by_0$ in $G'$}}.
\end{equation}
This implies that for every $\by_0 \in V \cup V'$ and $\by_0 \to \by \in E$,
\[
\by' - \by_0 \in S_{G'},
\]
and thus $S_G \subseteq S_{G'}$.
On the other hand, $G'$ is weakly reversible and \cite[Lemma 3.9]{craciun2021uniqueness} shows that for any $\bk' \in \mathbb{R}^{|E'|}_{>0}$,
\[
S_{G'} =
\spn \Big\{ \sum\limits_{\by_0 \to \by' \in E'} k'_{\by_0 \to \by'} (\by'- \by_0), 
\ \forall \by_0 \in V'
\Big\}.
\]
Together with \eqref{eq:w_y0=w'_y0}, this indicates that $S_{G'} \subseteq S_{G}$ and we conclude $S_{G} = S_{G'}$.

\medskip

Second, we prove $\New (G) = \New (G')$.
Since $G'$ is weakly reversible, \eqref{eq:w_y0=w'_y0} shows $V_{s} \subseteq V'$. Then it suffices to show that for every vertex $\by_0 \in V'$, then
\[
\by_0 \in \New (G). 
\]

We prove this by contradiction. Suppose there exist some vertices in $V'$ lying outside $\New (G)$.
Using Theorem~\ref{thm:separating_hyperplane} combined with the fact $G'$ is weakly reversible and has a single linkage class, there exists a vector $\bv \in \mathbb{R}^n$ and $\by_0 \to \by_1 \in E'$ such that $\bv \cdot (\by_1 - \by_0 ) < 0$ and 
\[
\bv \cdot \by_0 \leq \bv \cdot \by'
\ \text{ for every }
\by' \in V'.
\]
This implies that for any $\bk' \in \mathbb{R}^{|E'|}_{>0}$,
\begin{equation} \notag
\bw'_{\by_0} \cdot \bv = \Big( \sum\limits_{\by_0 \to \by' \in E'} k'_{\by_0 \to \by'} (\by'- \by_0) \Big) \cdot \bv \geq k'_{\by_0 \to \by_1} (\by_1 - \by_0) > 0,
\end{equation}
and thus $\bw'_{\by_0} \neq \mathbf{0}$. This contradicts with \eqref{eq:w_y0=w'_y0} since $\bw_{\by_0} = \mathbf{0}$ from $\by_0 \notin V$.
\end{proof}

We are now ready to present the main result of this section.

\begin{theorem}
\label{thm:our_theorem_2}

Let $G=(V, E)$ be a strongly endotactic 2D E-graph with a two-dimensional stoichiometric subspace. Then there exists a weakly reversible single linkage class E-graph $G'\neq G$ such that $G\sqsubseteq G'$ if and only if at least one of the following holds:
\begin{enumerate}[label=(\roman*)]
\item all source vertices of $G$ lie on boundary of $\New (G)$.

\item there exists a source vertex $\by_0$ on the boundary of $\New (G)$, such that the net reaction vector corresponding to $\by_0$ points strictly in the interior of $\New (G)$.
\end{enumerate}
\end{theorem}

\begin{proof}

$(\Longleftarrow)$ 
In case $(i)$, suppose all source vertices of $G$ lie on the boundary of $\New (G)$. Theorem \ref{thm:our_theorem} shows that there exists a weakly reversible E-graph $G'$ such that it has a single linkage class and $G\sqsubseteq G'$. 

In case $(ii)$, suppose some source vertices lie in the interior of $\New (G)$.
Since there exists a source vertex $\by_0$ on the boundary of $\New (G)$ such that the net reaction vector corresponding to $\by_0$ points strictly in the interior of $\New (G)$, Theorem \ref{thm:two_strong_endo_weak} shows that there exists a weakly reversible single linkage E-graph $G'$, such that $G\sqsubseteq G'$.

\smallskip

$(\Longrightarrow)$ 
Suppose that $G \sqsubseteq G'$ and $G'$ is a weakly reversible single linkage E-graph.
If all vertices lie on the boundary of $\New (G)$, then it belongs to the case $(i)$.

Otherwise, suppose some source vertices lie in the interior of $\New (G)$.
Since $G' = (V', E')$ is weakly reversible and has a single linkage class, then we have
\[
\by_0 \to \by_1 \in E',
\]
where $\by_0$ and $\by_1$ lying on the boundary of $\New (G)$ and in the interior of $\New (G)$, respectively.
From $G \sqsubseteq G'$, for every $\bk \in \mathbb{R}^{|E|}_{>0}$ there exists $\bk' \in \mathbb{R}^{|E'|}_{>0}$ such that
\begin{equation} \notag
\sum\limits_{\by_0 \to \by \in E}
k_{\by_0 \to \by'} ( \by' - \by_0 ) = \sum\limits_{\by_0 \to \by' \in E'} k'_{\by_0 \to \by'} (\by'- \by_0).
\end{equation}
From Lemma \ref{lem:newt_G_G'}, $G'$ possesses a two-dimensional stoichiometric subspace, together with $\by_0 \to \by_1 \in E'$, we obtain that
\[
\sum\limits_{\by_0 \to \by' \in E'} k'_{\by_0 \to \by'} (\by'- \by_0) 
\text{ points strictly in the interior of }
\New (G').
\]
Lemma \ref{lem:newt_G_G'} further shows that $\New (G) = \New (G')$ and thus the net reaction vector to $\by_0$ points strictly in the interior of $\New (G)$.
\end{proof}

\subsection{Endotactic Networks in Higher Dimensions}
\label{sec:general_strong_endo}

It is important to note that Theorem \ref{thm:our_theorem_2} is insufficient when applied to cases in higher dimensions ($\geq 3$).  
This limitation is illustrated in Figure~\ref{fig:3dexample}, which presents a counterexample where there is no weakly reversible E-graph whose dynamics can include the dynamics generated by this network, even though it satisfies the conditions of Theorem \ref{thm:our_theorem_2}.

In order to address this issue in higher-dimensional scenarios, a stronger sufficient condition is required, which is precisely provided by Theorem~\ref{thm:suff_strong_endo}.

\begin{figure}[!ht]
\centering
\includegraphics[width=0.4\textwidth]{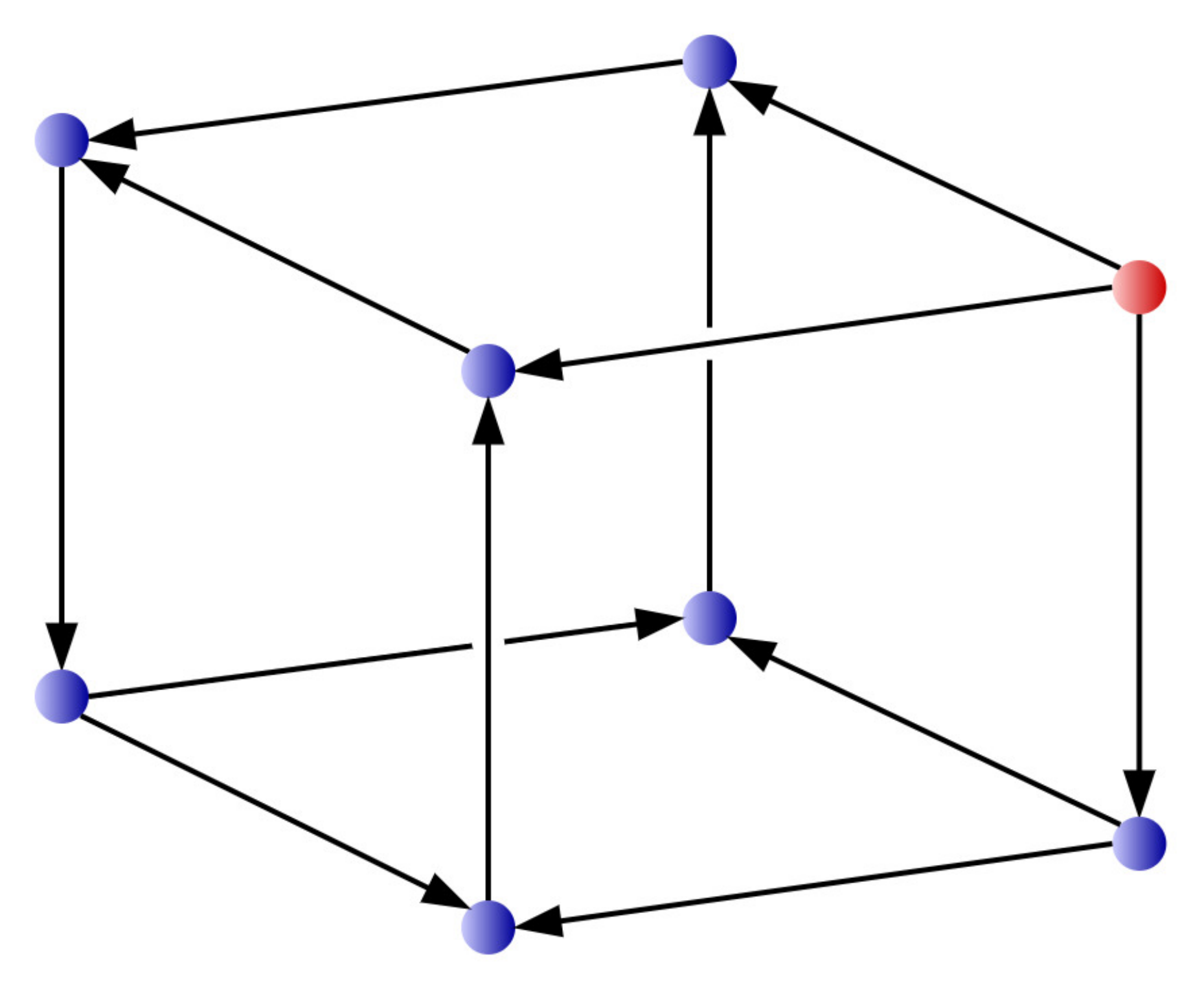}
\caption{
A strongly endotactic 3D E-graph $G$.
Note that the net reaction vector corresponding to the red vertex points strictly in the interior of $\New (G)$.
However, there is no weakly reversible E-graph whose dynamics can include the dynamics of this network, since there is no path that can return to the red vertex.}    
\label{fig:3dexample}
\end{figure}

\begin{theorem}
\label{thm:suff_strong_endo}

Let $G = (V, E)$ be an E-graph that has $\ell$ linkage classes, denoted by $L_1, \ldots, L_{\ell}$, and $p$ terminal strongly connected components, denoted by $T_{1}, \ldots, T_{p}$.
For every $\bk \in \mathbb{R}^{|E|}_{>0}$, every terminal strongly connected component $T_i$ contains a vertex whose net reaction vector points strictly in the interior of $\New (L_j)$ with $T_i \subset L_j$.
Then there exists a weakly reversible E-graph $G'$ such that $G \sqsubseteq G'$.

\end{theorem}

\begin{proof}

If every linkage class in $G$ is weakly reversible, then $G$ is weakly reversible. Let $G' = G$ and we have $G \sqsubseteq G'$. Thus, we assume that $G$ is not weakly reversible, i.e., $G$ has at least one non-weakly reversible linkage class.
Without loss of generality, we assume the linkage class $L_1 \subset G$ is not weakly reversible. 

In what follows, we will describe the procedure for the linkage class $L_1$. A similar procedure will be performed on every non-weakly reversible linkage class.
We set the initial E-graph $G'_1 = L_1$ which will be updated in the following process.
Note that $L_1$ can be partitioned by disjoint strongly connected components as
\begin{equation} \notag
L_1 = S_1 \cup S_2 \cup \cdots \cup S_{k_1}.
\end{equation}
Since $L_1$ is not weakly reversible, it must contain at least one terminal strongly connected component. Hence, we further assume $\{ S_1 = T_1, \cdots, S_i = T_i \}$ with $1 \leq i \leq k_1$ are terminal strongly connected components in $L_1$.

Given any $\bk \in \mathbb{R}^{|E|}_{>0}$, consider a terminal strongly connected component $T_1 \subset L_1$. From the assumption, there exists a vertex $\by \in T_1$ whose net reaction vector $\bw_{\by}$ points strictly in the interior of $\New (L_1)$.
Let $\{ \by_1, \ldots, \by_m \}$ denote the vertices in $L_1$, the assumption on $\bw_{\by}$ implies that there exists positive constants $\alpha_{1}, \ldots ,\alpha_{m}$, such that
\begin{equation} \notag
\bw_{\by} = \alpha_{1} (\by_{1} - \by) + \cdots + \alpha_{m} (\by_{m} - \by),
\end{equation}
where $\alpha_{1}, \ldots, \alpha_{m}$ depend on $\bk$.
Thus, the net reaction vector $\bw_{\by}$ can be realized as follows:
\begin{equation} \notag
\by \xrightarrow{\alpha_{j}} \by_{j}
\ \text{ for } \
j = 1, \ldots, m.
\end{equation}
We update $G'_1$ by adding edges from $\by$ to all other vertices in $G'_1$.
The above shows that $\bw_{\by}$ can be realized in the updated $G'$.

Analogously, we can do the same operation through the rest of the terminal strongly connected components $\{ T_2, \ldots, T_i \}$ in $L_1$.
Specifically, every terminal strongly connected component contains a vertex $\tilde{\by}$, such that its net reaction vector $\bw_{\tilde{\by}}$ satisfies
\begin{equation} \notag
\bw_{\tilde{\by}} = \tilde{\alpha}_{1} (\by_{1} - \tilde{\by}) + \cdots + \tilde{\alpha}_{m} (\by_{m} - \tilde{\by})
\ \text{ with } \
\tilde{\alpha}_{1}, \ldots \tilde{\alpha}_{m} > 0.
\end{equation}
Hence, $\bw_{\tilde{\by}}$ can be realized as follows:
\begin{equation} \notag
\tilde{\by} \xrightarrow{\tilde{\alpha}_{j}} \by_{j}
\ \text{ for } \
j = 1, \ldots, m.
\end{equation}
At every step when we work on a terminal strongly connected component, we update $G'_1$ by adding edges from such vertex $\tilde{\by}$ to all other vertices in $G'_1$.
From the construction, we deduce that 
\[
L_1 \sqsubseteq G'_1.
\]

It remains to prove that $G'_1$ is weakly reversible.
Note that $\{T_1, \ldots, T_i\}$ are terminal strongly connected components in $L_1$.
Thus, there is a path from every non-terminal strongly connected component to one of $\{T_1, \ldots, T_i\}$.
Furthermore, from the construction every component in $\{T_1, \ldots, T_i\}$ has direct edges to all other strongly connected components in $G'$, that is, for $1 \leq j \leq i$
\begin{equation} \notag
T_j \to S_l
\ \text{ for } \
1 \leq l \leq k_1, \ l \neq j.
\end{equation}
Therefore, $G'_1$ is weakly reversible and we prove this theorem.
\end{proof}

\begin{lemma}
\label{lem:hig_imply_endp}

Let $(G, \bk)$ be a mass-action system with $\ell$ linkage classes, denoted by $L_1, \ldots$, $L_{\ell}$, and $p$ terminal strongly connected components, denoted by $T_{1}, T_{2}, \ldots, T_{p}$.
Every terminal strongly connected component $T_i$ contains a vertex whose net reaction vector points strictly in the interior of $\New (L_j)$ with $T_i \subset L_j$. Then $G$ is endotactic.
\end{lemma}

\begin{proof}

We prove this lemma by contradiction. Assume that $G$ is not endotactic. 
From Definition \ref{def:endotactic}, there exists a reaction $\by \to \by' \in E$ and $\bv \in \mathbb{R}^n$ such that 
\begin{equation} \label{eq:v*y'-y<0_3}
\bv \cdot (\by'-\by) < 0.
\end{equation}
Further, for every $\tilde{\by} \to \tilde{\by}' \in E$, it satisfies 
\begin{equation} \label{eq:v*ty>=v*y_3}
\bv \cdot (\tilde{\by}'- \tilde{\by}) \leq 0
\ \text{ or } \
\bv \cdot \tilde{\by} \geq \bv \cdot \by.
\end{equation}

Suppose the vertex $\by \in L_1$, then we define 
\[
V_{\by} = \{ \tilde{\by} \in L_1 \ | \ \bv \cdot \tilde{\by} < \bv \cdot \by \}.
\]
From \eqref{eq:v*y'-y<0_3} and $\by \to \by' \in E$, we get $\by' \in V_{\by}$ and $\by \in L_1 \backslash V_{\by}$. This implies that
\[
V_{\by}  \neq \emptyset, \ \
L_1 \backslash V_{\by} \neq \emptyset.
\]
Moreover, \eqref{eq:v*ty>=v*y_3} shows that every vertex $\tilde{\by} \in V_{\by}$ satisfies that
\begin{equation} \label{eq:v*ty>=v*y_4}
\bv \cdot (\tilde{\by}'- \tilde{\by}) \leq 0
\ \text{ for any } \
\tilde{\by} \to \tilde{\by}' \in E.
\end{equation}

Hence, there is no reaction from the vertices in $V_{\by}$ to the vertices in $L_1 \backslash V_{\by}$. This indicates that there exists a terminal strongly connected component $T_i \subset V_{\by}$.
Consider any vertex $\tilde{\by} \in T_i$, from \eqref{eq:v*ty>=v*y_4} its net reaction vectors satisfies $\bv \cdot \bw_{\tilde{\by}} \leq 0$.
Recall that $L_1 \backslash V_{\by} \neq \emptyset$, this contradicts the assumption that $T_i$ contains a vertex whose net reaction vector points strictly in the interior of $\New (L_1)$.
Therefore, we conclude that $G$ is endotactic.
\end{proof}

\section{Toric Locus, Disguised Toric Locus, and Globally Attracting Locus}
\label{sec:toric_global_locus}

In Sections \ref{sec:strongly_endo_two} and \ref{sec:general_strong_endo}, we identified two types of E-graphs whose dynamics are included in those of weakly reversible E-graphs.
Several studies have demonstrated that weakly reversible E-graphs exhibit favorable dynamical properties, including persistence and permanence in low dimensions, and the existence of a steady state~\cite{boros2019existence,kothari2024endotactic}.

Building on this, we shift our focus to exploring whether the E-graphs analyzed in Section \ref{sec:sufficient} can admit a global attractor.
To this end, we introduce the concept of the \emph{globally attracting locus} corresponding to an E-graph.

\begin{definition} 

Consider an E-graph $G =(V, E)$. The \defi{globally attracting locus} of $G$ is defined as follows:
\begin{equation} \notag
\begin{split}
\gK(G)  := \big\{ \bk \in \mathbb{R}^{|E|}_{>0} \ \big| \ 
& (G, \bk) \text{ has a globally attracting steady state} 
\\& \ \ \text{within each stoichiometric compatibility class} \big\}.
\end{split}
\end{equation}
\end{definition}

Inspired by the global attractor conjecture, we are motivated to study the \emph{toric locus} and the \emph{disguised toric locus} corresponding to an E-graph.

\begin{definition} 

Consider an E-graph $G=(V, E)$.
\begin{enumerate}[label=(\roman*)]

\item The \defi{toric locus} on $G$ (denoted by $\mK (G)$) is defined as follows:
\begin{equation}\notag
\mK (G) := \{ \bk \in \mathbb{R}^{|E|}_{>0} \ \big| \ (G, \bk) \ \text{is a complex-balanced system} \}.
\end{equation}

\item Consider a dynamical system of the form 
\begin{equation} \notag
 \frac{d\bx}{dt} 
= \bof (\bx).
\end{equation}
This dynamical system is said to be \defi{disguised toric} on $G$
if it has a realization $(G, \bk)$ with $\bk \in \mK (G)$, that is, it has a \defi{complex-balanced realization} $(G, \bk)$.
\end{enumerate}
\end{definition}

We now introduce the notion of disguised toric locus~\cite{craciun2024connectivity,i2022disguised}.

\begin{definition} 
\label{def:de_realizable}

Consider two E-graphs $G =(V,E)$ and $G' =(V', E')$.
\begin{enumerate}[label=(\roman*)]
\item The set $\dK(G, G')$ is defined as follows:
\begin{equation} \notag
\dK(G, G') := \{ \bk \in \mathbb{R}^{|E|}_{>0} \ \big| \ (G, \bk) \ \text{is disguised toric on } G' \}.
\end{equation}

\item The \defi{disguised toric locus} of $G$ is defined  as follows:
\begin{equation} \notag
\dK(G) = \bigcup_{G' \subseteq_{wr} G_{c}} \ \dK(G, G').
\end{equation}
\end{enumerate}
\end{definition}

Based on the above definitions, the disguised toric locus represents a set of reaction rate vectors for which the mass-action system is dynamically equivalent to a complex-balanced system.
Since \cite{siegel2000global} has demonstrated that persistent complex-balanced systems possess a global attractor, it is logical to investigate the connection between the disguised toric locus and the globally attracting locus.
Before that, it is essential to understand the disguised toric locus of a network or, at the very least, verify that it is non-empty.

From Definition \ref{def:de_realizable}, in order to determine whether $\dK(G)\neq \emptyset$ it suffices to check if $\dK(G, G') \neq \emptyset$ for some weakly reversible graph $G' \subseteq G_{c}$.
This can be accomplished by solving the linear programming problem outlined below (see Linear program $(P2)$ for details).
To proceed, we first review a few definitions from \cite{craciun2023lower,craciun2020efficient}.

\begin{definition}[\cite{craciun2023lower}]
\label{def:flux_equiv}

Consider an E-graph $G=(V, E)$.

\begin{enumerate}[label=(\roman*)]
\item The \defi{flux vector} of $G$ is given by
\[
\bJ = (J_{\by \to \by'})_{\by \to \by' \in E} \in \mathbb{R}_{>0}^{|E|},
\]
where $J_{\by \to \by'} > 0$ is called the \defi{flux} of the reaction $\by \to \by'\in E$.
The \defi{associated flux system} generated by $(G, \bJ)$ is given by
\begin{equation} \notag
\frac{d\bx}{dt} = \sum_{\by \to \by' \in E} J_{\by \to \by'} (\by' - \by).
\end{equation}

\item Consider two flux systems $(G,\bJ)$ and $(G', \bJ')$. Then $(G,\bJ)$ and $(G', \bJ')$ are said to be \defi{flux equivalent} if for every vertex $\by_0\in V \cup V'$,
\begin{equation} \notag
\displaystyle \sum_{\by_0 \to \by_j \in E} J_{\by_0 \to \by_j} (\by_j - \by_0) 
 = \sum_{\by_0 \to \by'_j \in E'} J'_{\by_0 \to \by'_j} (\by'_j - \by_0).
\end{equation}
Further, let $(G, \bJ) \sim (G', \bJ')$ denote two flux systems $(G, \bJ)$ and $(G', \bJ')$ are flux equivalent. 
\end{enumerate}
\end{definition}

\begin{proposition}[\cite{craciun2020efficient}]
\label{prop:craciun2020efficient}

Consider two mass-action systems $(G, \bk)$ and $(G', \bk')$. Let $\bx \in \mathbb{R}_{>0}^n$, define the flux vector $\bJ (\bx) = (J_{\by \to \by'})_{\by \to \by' \in E}$ on $G$, such that for every $\by \to \by' \in E$,
\begin{equation} \notag
J_{\by \to \by'} = k_{\by \to \by'} \bx^{\by}.
\end{equation}
Further, define the flux vector $\bJ' (\bx) = (J'_{\by \to \by'})_{\by \to \by' \in E'}$ on $G'$, such that for every $\by \to \by' \in E'$,
\begin{equation} \notag
J'_{\by \to \by'} = k'_{\by \to \by'} \bx^{\by}.
\end{equation} 
Then the flux systems $(G, \bJ(\bx)) \sim (G', \bJ'(\bx))$ for some $\bx \in \mathbb{R}_{>0}^n$ if and only if the mass-action systems $(G, \bk) \sim (G', \bk')$.
\end{proposition}

We now present the linear program designed to determine whether $\dK(G) \neq \emptyset$. 

\medskip

\textbf{Linear program (P2):}
Given an E-graph $G$, consider its complete graph $G_c = (V, E_c)$. Find $\bJ = (J_{\by \to \by'})_{\by \to \by' \in E} \in \mathbb{R}^{|E|}_{>0}$ and $\bJ' = (J'_{\by \to \by'})_{\by \to \by' \in E_c} \in \mathbb{R}^{|E_c|}_{\geq 0}$ satisfying for every $ \by_0 \in V$,
\begin{align}
\displaystyle \sum_{\by_0 \to \by_i \in E} J_{\by_0 \to \by_i} (\by_i - \by_0) 
& = \sum_{\by_0 \to \by_j \in E_c} J'_{\by_0 \to \by_j} (\by_j - \by_0),
\label{eq:flux_equivalence} \\
\sum_{\by_0 \to \by \in E_c} J'_{\by_0 \to \by} 
& = \displaystyle\sum_{\by' \to \by_0 \in E_c} J'_{\by' \to \by_0}.
\label{eq:complex_balance_flux}
\end{align}

\medskip

\begin{lemma}
\label{lem:linear_program}

Let $G$ be an E-graph.
If the linear program $(P2)$ has a solution, then there exists a weakly reversible realization E-graph $G' \subseteq_{wr} G_c$ such that $\dK(G, G') \neq \emptyset$. Otherwise, if the linear program $(P2)$ has no solution, then $\dK(G)= \emptyset$.
\end{lemma}

\begin{proof}

Suppose the linear program $(P2)$ has a solution, there exist $\bJ \in \mathbb{R}^{|E|}_{>0}$ and $\bJ' \in \mathbb{R}^{|E_c|}_{\geq 0}$ satisfying \eqref{eq:flux_equivalence} and \eqref{eq:complex_balance_flux}.
By selecting the edges with positive entries in $\bJ'$, we obtain an E-graph $G' \subseteq G_c$. In other word, for every $\by \to \by' \in G'$, $J'_{\by \to \by'} > 0$.

Consider $\bk = (J_{\by \to \by'})_{\by \to \by' \in E}$ as a reaction rate vector of $G$ and $\bk' = (J'_{\by \to \by'})_{\by \to \by' \in G'}$ as a reaction rate vector of $G'$.
Applying Proposition \ref{prop:craciun2020efficient} with $\bx^* = (1, \ldots, 1 )^{\intercal}$, together with \eqref{eq:flux_equivalence}, we derive that $(G, \bk) \sim (G',\bk')$.
Furthermore, \eqref{eq:complex_balance_flux} and Definition \ref{def:cb} show that $(G', \bk')$ is a complex-balanced system with the complex-balanced steady state $\bx^*$.
Therefore, we conclude that $\dK(G, G') \neq \emptyset$ and thus $\dK(G)\neq \emptyset$.

Conversely, if the linear program $(P2)$ has no solution, we claim that $\dK(G)= \emptyset$.
Suppose not, there exists $G' \subseteq_{wr} G_c$ such that $\bk \in \dK(G, G') \neq \emptyset$. This implies the existence of $\bk' \in \mK (G')$ such that $(G, \bk) \sim (G',\bk')$.

Moreover, $(G', \bk')$ admits a complex-balanced steady state $\bx^*$. 
Consider the flux vector $\bJ = (k_{\by \to \by'} (\bx^*)^{\by})_{\by \to \by' \in E}$ on $G$ and the flux vector $\bJ' = (k'_{\by \to \by'} (\bx^*)^{\by})_{\by \to \by' \in E'}$ on $G'$.
From Proposition \ref{prop:craciun2020efficient}, together with Definition \ref{def:cb}, we derive that $(G, \bJ) \sim (G',\bJ')$ and thus \eqref{eq:flux_equivalence}, \eqref{eq:complex_balance_flux} hold for $\bJ, \bJ'$.
This contradicts the fact that the linear program $(P2)$ has no solution.
Therefore, we conclude $\dK(G)= \emptyset$.
\end{proof}

We are now ready to present the main result of this section, which pertains to networks closely related to those analyzed in Sections~\ref{sec:strongly_endo_two} and~\ref{sec:general_strong_endo}. 

\begin{theorem}
\label{thm:global_two_dim}

Let $G=(V, E)$ be an endotactic E-graph with a two-dimensional stoichiometric subspace. Assume that the linear program $(P2)$ has a solution, then $\gK (G) \neq \emptyset$.
\end{theorem}

\begin{proof}

Since the linear program $(P2)$ has a solution, there exists $G'\subseteq_{wr} G_{c}$ such that 
\[
\dK(G,G') \neq \emptyset.
\]
For any $\bk \in \dK(G,G')$, there exists $\bk' \in \mK(G')$ such that $(G, \bk) \sim (G', \bk')$. 
This shows that $(G', \bk')$ is complex-balanced and $(G, \bk)$ is disguised toric.

Since $G'$ is weakly reversible, \cite[Lemma 3.9]{craciun2021uniqueness} shows that for any $\bk' \in \mathbb{R}^{|E'|}_{>0}$,
\[
S_{G'} =
\spn \Big\{ \sum\limits_{\by_0 \to \by' \in E'} k'_{\by_0 \to \by'} (\by'- \by_0), 
\ \forall \by_0 \in V'
\Big\}.
\]
Together with $(G,\bk_0)\sim (G',\bk')$ and $\dim (S_{G}) = 2$, we derive that
\[
\dim (S_{G'}) \leq 2.
\]
Moreover, due to the existence of the Lyapunov function \cite{horn1972general}, all trajectories of $(G', \bk')$ are bounded and converge either to the boundary of the positive orthant or to the unique positive equilibrium within the corresponding stoichiometric compatibility classes.

On the other hand, \cite{pantea2012persistence} shows weakly reversible networks with one- or two-dimensional stoichiometric subspaces and bounded trajectories are persistent.
Consequently, $(G', \bk')$ is persistent and every trajectory converges to the unique equilibrium within each stoichiometric compatibility class. This indicates that $(G', \bk')$ has a globally attracting steady state within each stoichiometric compatibility class. Since $(G, \bk) \sim (G', \bk')$, we conclude that $\dK(G, G') \subseteq \gK(G)$ and thus $\gK(G) \neq \emptyset$.
\end{proof}



\begin{corollary}
\label{cor:global_strongly_endotactic}

Let $G=(V, E)$ be a strongly endotactic E-graph. Assume that the linear program $(P2)$ has a solution, then $\gK (G) \neq \emptyset$.
\end{corollary}

\begin{proof}

Similar to the proof of Theorem \ref{thm:global_two_dim}, there exists $G'\subseteq_{wr} G_{c}$ such that $\dK(G,G') \neq \emptyset$.
For any $\bk \in \dK(G,G')$, there exists $\bk' \in \mK(G')$ such that $(G, \bk) \sim (G', \bk')$, where $(G', \bk')$ is complex-balanced and $(G, \bk)$ is  disguised toric. 
As a complex-balanced system, all trajectories of $(G', \bk')$ converge either to the boundary of the positive orthant or to the unique positive equilibrium within the corresponding stoichiometric compatibility classes.

On the other hand, $G$ is strongly endotactic. \cite{gopalkrishnan2014geometric} shows that the dynamics generated by $(G,\bk)$ is permanent. 
Since $(G, \bk) \sim (G', \bk')$, it follows that $(G', \bk')$ is also permanent
Therefore, $(G', \bk')$ has a globally attracting steady state within each stoichiometric compatibility class, implying that $\gK(G) \supseteq \dK(G, G') \neq \emptyset$.
\end{proof}

\section{Applications}
\label{sec:application_positive}

In this section, we present several real-life examples of networks inspired by chemical and biological applications. We demonstrate that the dynamics of these networks can be included in the dynamics of weakly reversible networks. Additionally, we explore the disguised toric locus and the globally stable locus for these networks.

\begin{example}[Thomas type models {\cite[Chapter 6]{murray2002introduction}}]
\label{ex:thomas}

Consider the network $G$ shown in Figure~\ref{fig:thomas}(a). This network is an example of the \emph{Thomas type models}, commonly used to model the oxidation of uric acid by oxygen in the presence of the enzyme uricase. In this context, the species $X$ and $Y$ represent uric acid and oxygen, respectively.

\begin{figure}[!ht]
\centering
\includegraphics[scale=0.65]{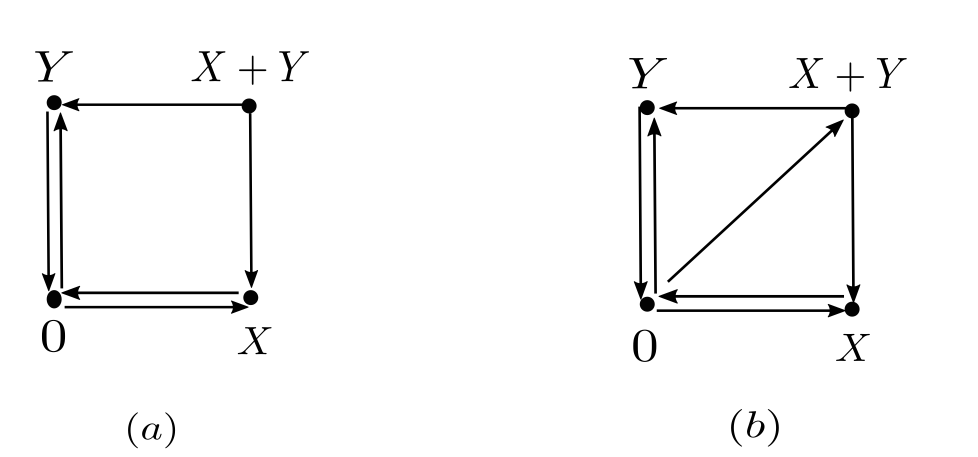}
\caption{(a) An E-graph $G$ represents the Thomas type model. (b) The weakly reversible E-graph $G'$ includes the dynamics of the network $G$ in (a).}
\label{fig:thomas}
\end{figure} 

Let $x$ and $y$ denote the concentration of uric acid and oxygen, respectively. The dynamics of this network is given by
\begin{equation} \notag
\begin{split}
\frac{d x}{d t} & = a - x - K xy,
\\ \frac{d y}{d t} & = b - y - K xy,
\end{split}
\end{equation}
where $a, b$ are positive parameters and $K$ represents the strength of \emph{substrate inhibition}.
Note that the network $G$ is strongly endotactic and all source vertices $\{ \emptyset, X, Y, X+Y \}$ lie on the boundary of the Newton polytope $\New (G)$. 
From Theorem \ref{thm:our_theorem_2}, there exists a weakly reversible single linkage E-graph $G'$, as illustrated in Figure~\ref{fig:thomas}(b), such that $G\sqsubseteq G'$. 

We now consider the linear program (P2), which has a solution as follows:
\begin{equation}
\begin{split} \notag
& J'_{0 \to X} = J'_{0 \to Y} = J'_{X+Y \to X} = J'_{X+Y \to Y} = 1, \ J'_{X \to 0} = J'_{Y \to 0} = J'_{0 \to X+Y} = 2,
\\& J_{0 \to X} = J_{0 \to Y} = 3, \ J_{X+Y \to X} = J_{X+Y \to Y} = 1, \ J_{X \to 0} = J_{Y \to 0} = 2.
\end{split}
\end{equation}
From Lemma \ref{lem:linear_program}, we obtain that $\dK(G, G') \neq \emptyset$.
Theorem \ref{thm:global_two_dim} further shows that $\dK(G)\subseteq \gK(G)$, therefore we conclude $\gK(G)\neq\emptyset$.
\end{example}

\begin{example}[Modified Selkov models {\cite{boros2022limit,brusselator,prigogine1968symmetry}}]
\label{ex:selkov}

Consider the network $G$ shown in Figure~\ref{fig:bruseelator}(a). This network is an example of a modified \emph{Selkov model}, commonly used to model glycolysis, a multi-step anaerobic process in which glucose is broken down into pyruvate.
A detailed analysis of the dynamics of this network, as a function of their parameters, reveals that it can exhibit complex behaviors such as chaos, oscillations, bifurcations, and turning instability.

\begin{figure}[!ht]
\centering
\includegraphics[scale=0.45]{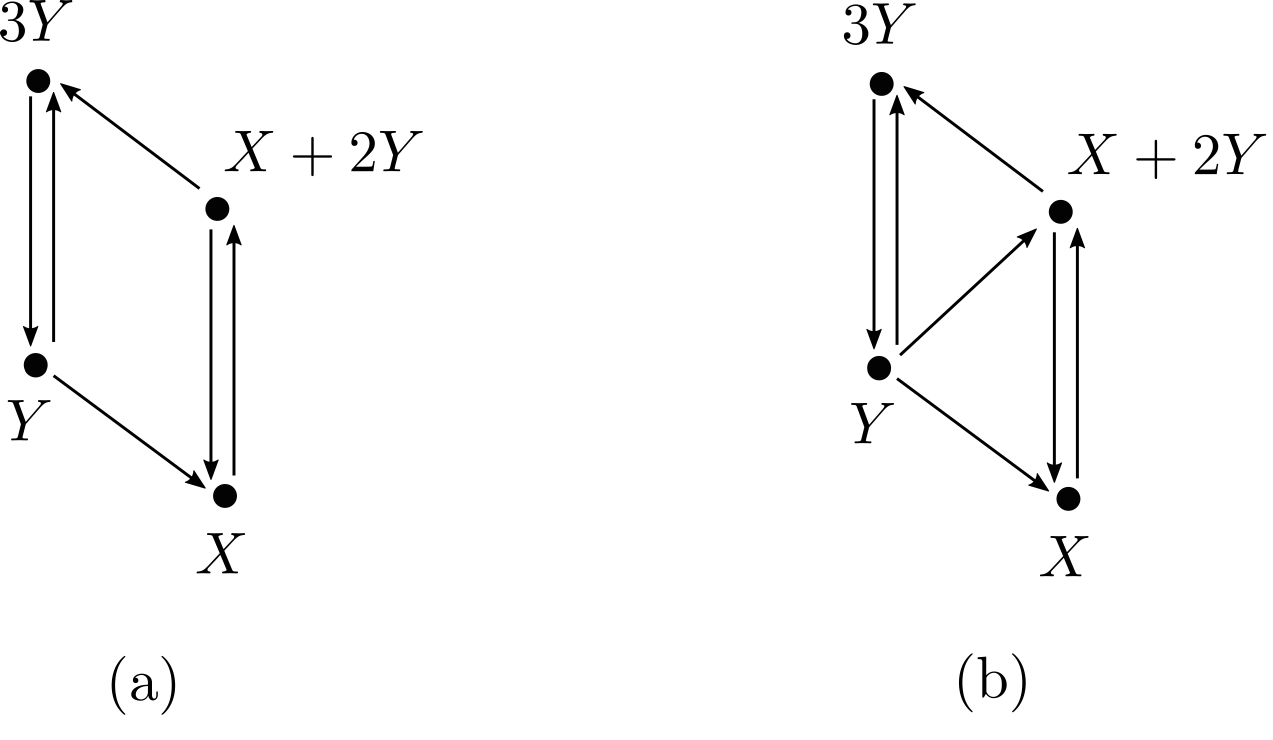}
\caption{(a) An E-graph $G$ represents the modified Selkov network \cite{boros2022limit}. (b) The weakly reversible E-graph $G'$ includes the dynamics of the network $G$ in (a).}
\label{fig:bruseelator}
\end{figure} 

Note that the network $G$ is strongly endotactic and all source vertices $\{ Y, X, 3Y, X+2Y \}$ lie on the boundary of the Newton polytope $\New (G)$. 
From Theorem \ref{thm:our_theorem}, there exists a weakly reversible single linkage E-graph $G'$, as illustrated in Figure~\ref{fig:bruseelator}(b), such that $G\sqsubseteq G'$.

We now consider the linear program (P2), which has a solution as follows:
\begin{equation}
\begin{split} \notag
& J'_{Y \to 3Y} = J'_{Y \to X + 2Y} = J'_{Y \to X} = J'_{X + 2Y \to X} = 1, \ J'_{X + 2Y \to 3Y} = J'_{X \to X + 2Y} = 2, \ J'_{3Y \to Y} = 3,
\\& J_{X + 2Y \to X} = 1, \ J_{Y \to 3Y} = J_{Y \to X} = J_{X + 2Y \to 3Y} = J_{X \to X + 2Y} = 2, \ J_{3Y \to Y} = 3.
\end{split}
\end{equation}
From Lemma \ref{lem:linear_program}, we obtain that $\dK(G, G') \neq \emptyset$.
Theorem \ref{thm:global_two_dim} further shows that $\dK(G)\subseteq \gK(G)$, therefore we conclude $\gK(G)\neq\emptyset$.
\end{example}

\section{Discussion}
\label{sec:disc}

In this paper, we derive conditions that enable us to establish the properties of dynamical inclusions between weakly reversible, endotactic, and strongly endotactic networks, assuming mass-action kinetics.
Our results extend several previous findings, particularly those discussed in~\cite{anderson2020classes}. 

A key idea in this context is that of \emph{Dynamical Equivalence} whereby it is possible for two different mass-action systems to exhibit the same dynamics. This phenomenon has a long history starting from the seminal work of Horn and Jackson \cite{horn1972general} to the more recent work of Craciun and Pantea~\cite{craciun2008identifiability}. Craciun, Jin, and Yu have established that for a mass-action system to be dynamically equivalent to a weakly reversible mass-action system, it suffices to consider only the source vertices of the original network~\cite{craciun2020efficient}. This has been extended by Deshpande~\cite{deshpande2023source} to strongly endotactic, endotactic, consistent, and conservative networks. Further Craciun, Deshpande, and Jin have used its properties to design efficient algorithms for detecting weakly reversible single linkage class realizations and weakly reversible deficiency one realizations~\cite{craciun2023weakly,craciun2023weaklysingle}.

Taking our cue from this, we establish sufficient conditions for the dynamics of a two-dimensional strongly endotactic network to be included in the dynamics of a weakly reversible network consisting of a single linkage class (see Section \ref{sec:strongly_endo_two}).
As a consequence, we examine the globally stable locus of the networks discussed in Section \ref{sec:sufficient} in Section \ref{sec:toric_global_locus}. In particular, we demonstrate that their locus is non-empty if solutions exist for the linear program (P2).
As pointed out in the Examples \ref{ex:thomas} and \ref{ex:selkov}, our results are particularly relevant for Thomas and Selkov models, which have significant applications in chemistry and biology. By investigating the disguised toric locus introduced in Section \ref{sec:toric_global_locus}, we can examine the globally stable locus of these models by analyzing the structure of the underlying reaction network.

While beyond the scope of our current work, the approach proposed here has potential applications for poly-exponential systems derived from reaction networks. A poly-exponential system generated by $(G, \bk)$ is the dynamical system on $\mathbb{R}_{>0}^n$ defined by
\begin{equation} \notag
\frac{d \bx}{d t} = \sum_{\by\rightarrow \by'} k_{\by\rightarrow \by'} e^{\langle \bx, \by \rangle} (\by' - \by).
\end{equation}
Based on the strong connections between poly-exponential systems and generalized Lotka-Volterra models, which are commonly used to describe the dynamics of biological systems, it is of interest to identify the conditions under which poly-exponential systems exhibit a non-empty globally stable locus. Some relevant ideas have been recently proposed in this area by \cite{poly_exp_2017}. 

Besides poly-exponential systems, there are other applications where our theory could be applied.
The first one is inspired by recent findings on the disguised toric locus \cite{craciun2024connectivity, craciun2023lower}. As discussed in this paper, the globally stable locus of certain networks is non-empty if their disguised toric locus is non-empty. In~\cite{craciun2023lower}, a lower bound on the dimension of the disguised toric locus of a reaction network is provided. Consequently, it is valuable to investigate the dimension of the globally stable locus in these networks.
Furthermore, our goal is to establish sufficient conditions for a network to possess a globally stable locus of full dimension within the parameter space.

Another area for future exploration is the study of dynamical inclusions between weakly reversible, endotactic, source-only, and target-only networks. It is known that the dynamics of an endotactic network are contained within the dynamics of a source-only network~\cite{anderson2020classes}. Extending this result to more generalized networks seems promising. Specifically, it would be valuable to classify the family of networks whose dynamics are contained within those of target-only networks. We hope to pursue this line of thought in future work.

\bibliographystyle{amsplain}
\bibliography{Bibliography}

\newpage

\appendix

\begin{appendices}

\section{Feasibility Problem}
\label{sec:p1}

Recall the question from Section \ref{sec:dyn_equiv}, we introduce the following nonlinear feasibility problem:

\medskip

\textbf{Feasibility problem (P1):}
Given an E-graph $G$, consider its complete graph $G_c = (V, E_c)$. Find an E-graph $G' \subseteq G_c$, such that for any $\bk \in \mathbb{R}^{|E|}_{>0}$ there exist 
\[
\bk' = (k'_{\by \rightarrow \by'})_{\by \rightarrow \by' \in E'} \in \mathbb{R}^{|E'|}_{>0}
\text{ and }
\bzeta = (\zeta_{\by \rightarrow \by'})_{\by \rightarrow \by' \in E'} \in \mathbb{R}^{|E'|}_{>0},
\]
satisfying for every $ \by_0 \in V$,
\begin{align}
\displaystyle \sum_{\by_0 \to \by_i \in E} k_{\by_0 \to \by_i} (\by_i - \by_0) 
& = \displaystyle \sum_{\by_0 \to \by_j \in E'} k'_{\by_0 \to \by_j} (\by_j - \by_0), 
\label{eq:flux_equiv_non} \\
\sum_{\by_0 \to \by \in E'} \zeta_{\by_0 \to \by} k'_{\by_0 \to \by} 
& = \displaystyle\sum_{\by' \to \by_0 \in E'} \zeta_{\by' \to \by_0} k'_{\by' \to \by_0}.
\label{eq:complex_balance}
\end{align}

\begin{lemma}

Let $G$ be an E-graph.
If the feasibility problem $(P1)$ has a solution, then there exists a weakly reversible realization E-graph $G' \subseteq_{wr} G_c$ such that $G \sqsubseteq G'$.
\end{lemma}

\begin{proof}

From the assumption, Definition \ref{def:dyn_equ} and \eqref{eq:flux_equiv_non} imply that for every $\bk\in\mathbb{R}^{|E|}_{>0}$, there exists $\bk' \in \mathbb{R}^{|E'|}_{>0}$ such that $(G,\bk) \sim (G',\bk')$. Hence, $G \sqsubseteq G'$.

On the other hand, consider the reaction rate vector $\hat{\bk} = (k'_{\by \rightarrow \by'} \zeta_{\by \rightarrow \by'})_{\by \rightarrow \by' \in E'} \in \mathbb{R}^{|E'|}_{>0}$ of $G'$. From \eqref{eq:complex_balance}, we can compute that $(G', \hat{\bk})$ is a complex-balanced system with the complex-balanced steady state $\bx_0 = (1, \ldots,1)$.
From Remark \ref{rmk:cb_property}, $G'$ is weakly reversible. Therefore, we conclude that $G' \subseteq_{wr} G_c$.
\end{proof}

\begin{remark}

From \cite{craciun2020efficient}, it has been proved that given an E-graph $G$, it suffices to focus on weakly reversible subgraphs of $G_c$.
Specifically, if there exists a weakly reversible E-graph $\tilde{G} \not\subseteq G_c$ that includes the dynamics of $G$, then there exists a corresponding weakly reversible subgraph $G' \subseteq_{wr} G_c$ such that $G \sqsubseteq G'$.
\end{remark}

\end{appendices}

\end{document}